\documentclass{amsart}
\usepackage{amsmath,amsfonts,amsthm,amssymb,color,graphicx,subfigure,url}
 \theoremstyle{plain}
   \newtheorem{theorem}{Theorem}[section]
   \newtheorem*{thm216}{Theorem 2.6}
   \newtheorem*{thm217}{Theorem 2.7}
   
   \newtheorem*{thm219}{Theorem 2.9}
   \newtheorem*{thm2110}{Theorem 2.10}
   \newtheorem{proposition}[theorem]{Proposition}
   
   \newtheorem{corollary}[theorem]{Corollary}
   
\theoremstyle{definition}
   \newtheorem{definition}[theorem]{Definition}
\usepackage{fullpage}
\begin{document}
\author{Jessica Striker}
\title{A unifying poset perspective on alternating sign matrices, plane partitions, Catalan objects, tournaments, and tableaux}
\address{School of Mathematics, University of Minnesota, Minneapolis, MN 55455}
 
\begin{abstract}
Alternating sign matrices (ASMs) are square matrices with entries 0, 1, or $-1$ whose rows and columns sum to 1 and whose nonzero entries alternate in sign. We present a unifying perspective on ASMs and other combinatorial objects by studying a certain tetrahedral poset and its subposets. We prove the order ideals of these subposets are in bijection with a variety of interesting combinatorial objects, including ASMs, totally symmetric self-complementary plane partitions (TSSCPPs), staircase shaped semistandard Young tableaux, Catalan objects, tournaments, and totally symmetric plane partitions. We prove product formulas counting these order ideals and give the rank generating function of some of the corresponding lattices of order ideals. We also prove an expansion of the tournament generating function as a sum over TSSCPPs. This result is analogous to a result of Robbins and Rumsey expanding the tournament generating function as a sum over alternating sign matrices.
\end{abstract}

%

\maketitle

\section{Background and terminology}

\begin{definition}
\label{def:asm}
Alternating sign matrices (ASMs) are square matrices with the following properties:
\begin{itemize}
\item entries $\in \{0,1,-1\}$
\item the entries in each row and column sum to 1
\item nonzero entries in each row and column alternate in sign
\end{itemize}
\end{definition}
See Figure~\ref{fig:n3asm} for the seven ASMs with three rows and three columns.

\begin{figure}[htbp]
\[
\left( 
\begin{array}{rrr}
1 & 0 & 0 \\
0 & 1 & 0\\
0 & 0 & 1
\end{array} \right)
\left( 
\begin{array}{rrr}
1 & 0 & 0 \\
0 & 0 & 1\\
0 & 1 & 0
\end{array} \right)
\left( 
\begin{array}{rrr}
0 & 1 & 0 \\
1 & 0 & 0\\
0 & 0 & 1
\end{array} \right)
\left( 
\begin{array}{rrr}
0 & 1 & 0 \\
1 & -1 & 1\\
0 & 1 & 0
\end{array} \right)
\left( 
\begin{array}{rrr}
0 & 1 & 0 \\
0 & 0 & 1\\
1 & 0 & 0
\end{array} \right)
\left( 
\begin{array}{rrr}
0 & 0 & 1 \\
1 & 0 & 0\\
0 & 1 & 0
\end{array} \right)
\left( 
\begin{array}{rrr}
0 & 0 & 1 \\
0 & 1 & 0\\
1 & 0 & 0
\end{array} \right)
\]
\label{fig:n3asm}
\caption{The $3\times 3$ ASMs}
\end{figure}

In 1983 Mills, Robbins, and Rumsey conjectured that the total number of $n \times n$ alternating sign matrices is given by the expression
\begin{equation}
\label{eq:product}
\displaystyle\prod_{j=0}^{n-1} \frac{(3j+1)!}{(n+j)!}.
\end{equation}
They were unable to prove this for all $n$, so for the next 13 years it remained a mystery until Doron Zeilberger proved it~\cite{ZEILASM}. Shortly thereafter, Greg Kuperberg gave a shorter proof using a bijection between ASMs and configurations of the statistical physics model of square ice with domain wall boundary conditions~\cite{KUP_ASM_CONJ}. This connection with physics has since strengthened with the conjecture of Razumov and Stroganov which says that the enumeration of subclasses of ASMs gives the ground state probabilities of the dense O(1) loop model in statistical physics~\cite{O1LOOP}. This conjecture has been further refined and various special cases have been proved (see for example~\cite{DIFRANCESCOII}), but no general proof is known.

Expression (\ref{eq:product}) also counts totally symmetric self-complementary plane partitions in a $2n\times 2n\times 2n$ box~\cite{ANDREWS_PPV}.  
A \emph{plane partition} is a two-dimensional array of positive integers which weakly decrease across rows from left to right and weakly decrease down columns.
Equivalently, a plane partition $\pi$ is a finite set of positive integer lattice points $(i,j,k)$ such that if $(i,j,k)\in\pi$ and $1\le i'\le i$, $1\le j'\le j$, and $1\le k'\le k$ then $(i',j',k')\in\pi$. One can visualize this as stacks of unit cubes pulled toward the corner of a room.
A plane partitions is \emph{totally symmetric} 
if whenever $(i,j,k)\in\pi$ then all six permutations of $(i,j,k)$ are also in $\pi$.
\begin{definition}
\label{def:tsscpp}
A totally symmetric self-complementary plane partition (TSSCPP) inside a 
$2n \times 2n \times 2n$ box is a totally symmetric plane partition which is
equal to its complement in the sense that the collection of empty cubes in the box is of the same shape as the collection of cubes in the plane partition itself.  
\end{definition}
See Figure \ref{fig:cl10} for the $7$ TSSCPPs inside a $6\times 6\times 6$ box. 

\begin{figure}[htbp]
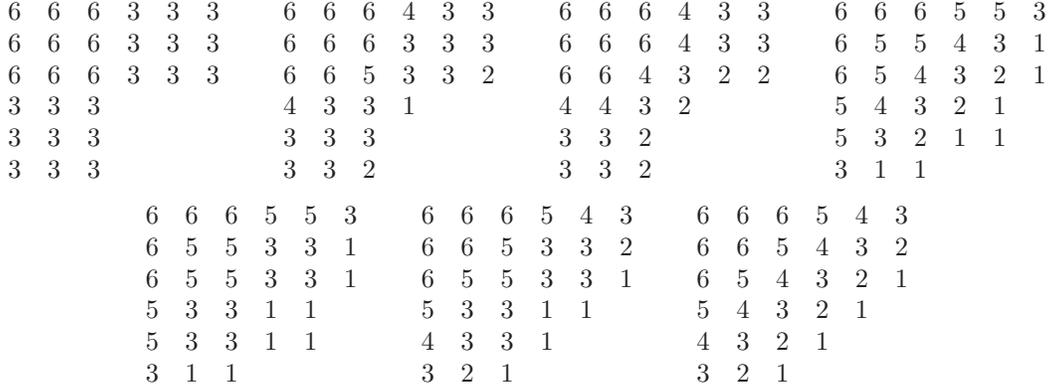

\[
\begin{array}{cccccc}
6&6&6&3&3&3\\
6&6&6&3&3&3\\
6&6&6&3&3&3\\
3&3&3&&&\\
3&3&3&&&\\
3&3&3&&&
\end{array}
\hspace{.5cm}
\begin{array}{cccccc}
6&6&6&4&3&3\\
6&6&6&3&3&3\\
6&6&5&3&3&2\\
4&3&3&1&&\\
3&3&3&&&\\
3&3&2&&&
\end{array}
\hspace{.5cm}
\begin{array}{cccccc}
6&6&6&4&3&3\\
6&6&6&4&3&3\\
6&6&4&3&2&2\\
4&4&3&2&&\\
3&3&2&&&\\
3&3&2&&&
\end{array}
\hspace{.5cm}
\begin{array}{cccccc}
6&6&6&5&5&3\\
6&5&5&4&3&1\\
6&5&4&3&2&1\\
5&4&3&2&1&\\
5&3&2&1&1&\\
3&1&1&&&
\end{array}
\]
\[
\begin{array}{cccccc}
6&6&6&5&5&3\\
6&5&5&3&3&1\\
6&5&5&3&3&1\\
5&3&3&1&1&\\
5&3&3&1&1&\\
3&1&1&&&
\end{array}
\hspace{.5cm}
\begin{array}{cccccc}
6&6&6&5&4&3\\
6&6&5&3&3&2\\
6&5&5&3&3&1\\
5&3&3&1&1&\\
4&3&3&1&&\\
3&2&1&&&
\end{array}
\hspace{.5cm}
\begin{array}{cccccc}
6&6&6&5&4&3\\
6&6&5&4&3&2\\
6&5&4&3&2&1\\
5&4&3&2&1&\\
4&3&2&1&&\\
3&2&1&&&
\end{array}
\]
\caption{TSSCPPs inside a $6\times 6\times 6$ box}
\label{fig:cl10}
\end{figure}

In~\cite{ANDREWS_PPV} Andrews showed that TSSCPPs inside a $2n\times 2n\times 2n$ box are counted by~(\ref{eq:product}). This proves that
TSSCPPs inside a $2n \times 2n \times 2n$ box are equinumerous with $n\times n$ ASMs, but no explicit bijection between these two sets of objects is known. In this paper we present a new perspective which shows that these two very different sets of objects arise naturally as members of a larger class of combinatorial objects. This perspective may bring us closer to the construction of an explicit ASM--TSSCPP bijection.


In order to establish this larger context 
we construct a tetrahedral partially ordered set (poset) containing subposets corresponding to each of these objects. We 
will follow the poset terminology from Chapter 3 of~\cite{STANLEY} with the notable exception that we will extend the definition of a Hasse diagram slightly by sometimes drawing edges in the Hasse diagram from $x$ to $y$ when $x<y$ but $y$ does not cover $x$. We will denote the rank generating function of a poset $P$ as $F(P,q)$. Also, we will use $P^*$ to mean the poset dual to $P$ and $J(P)$ to mean the lattice of order ideals of $P$. See~\cite{STRIKER_THESIS} for a more detailed summary of the poset definitions used in this paper. Also, for an extended abstract of this paper, see the proceedings article~\cite{STRIKER_FPSAC_09}.

The paper is organized as follows. In Section~\ref{sec:intro} we introduce the tetrahedral poset and list the theorems counting the order ideals of its subposets. Then in Sections~\ref{sec:two}--\ref{sec:four}, \ref{sec:five}, and~\ref{sec:six} we prove the theorems listed in Section~\ref{sec:intro} and discuss connections with combinatorial objects. Section~\ref{sec:3to4} uses this poset perspective to prove an expansion of the tournament generating function as a sum over TSSCPPs. This is analogous to the Robbins-Rumsey expansion of the tournament generating function as a sum over ASMs. Finally, Section~\ref{sec:trap} discusses a possible extension of the poset perspective presented in this paper.

\section{Introduction to the tetrahedral poset}
\label{sec:intro}
We now define posets $P_n$ and $T_n$ using certain unit vectors in $\mathbb{R}^3$. We also state the theorems which will be proved and explained throughout the rest of this paper concerning the order ideals of subposets of $P_n$ and $T_n$ and their connections to well-known combinatorial objects such as ASMs, TSSCPPs, Catalan objects, tournaments, and semistandard Young tableaux. We then discuss bijections between the order ideals of subposets of $T_n$ and certain integer arrays. 

We begin by constructing the pyramidal poset $P_n$ by use of certain unit vectors. Define the vectors $\overrightarrow{r}=(\frac{\sqrt{3}}{2},\frac{1}{2},0)$, $\overrightarrow{g}=(0,1,0)$, and $\overrightarrow{b}=(-\frac{\sqrt{3}}{2},\frac{1}{2},0)$. We use these vectors to define $P_n$ by drawing its Hasse diagram. First let the elements of $P_n$ be defined as the coordinates of all the points reached by linear combinations of $\overrightarrow{r}$ and $\overrightarrow{g}$. That is, as a set $P_n =\{c_1 \overrightarrow{r} +c_2 \overrightarrow{g}, \mbox{ } c_1,c_2\in \mathbb{Z}_{\ge 0}, \mbox{ } c_1+c_2\le n-2\}$. To obtain the partial order on $P_n$ let all the vectors $\overrightarrow{r}$ and $\overrightarrow{g}$ used to define the elements of $P_n$ be directed edges in the Hasse diagram, and additionally draw into the Hasse diagram as directed edges the vectors $\overrightarrow{b}$ between poset elements wherever possible.
See Figure~\ref{fig:p4} for an example of this construction. Thus it is easy to check that the Hasse diagram of $P_n$ has ${n\choose 2}$ 
vertices and $3 {n-1 \choose 2}$ edges. 

It will be useful for what follows to count the number of order ideals $|J(P_n)|$ and find the rank generating function $F((J(P_n),q))$.
\begin{theorem}
\label{lemma:pn}
\[|J(P_n)|=2^{n-1}. \hspace{.35in} F(J(P_n),q)=\displaystyle\prod_{j=1}^{n-1} (1+q^j).\]
\end{theorem}
\begin{proof}
The proof is by induction on $n$. Suppose $P_{n-1}$ has $2^{n-2}$ order ideals and rank generating function $\prod_{j=1}^{n-2} (1+q^j)$. Let $a$ be the element of $P_n$ with the largest $x$ coordinate. Thus $a=(n-2)\overrightarrow{r}$. Let $I$ be an order ideal of $P_n$. If $a\in I$ then the $n-2$ poset elements $c\overrightarrow{r}$ with $c<n-2$ are also in $I$, leaving a copy of $P_{n-1}$ from which to choose more elements to be in $I$ (and yielding a weight of $q^{n-1}$ for the principle order ideal of $a$). So there are $2^{n-1}$ order ideals of $P_n$ which include $a$. Now suppose $a\notin I$. Then any element larger than $a$ in the partial order is also not in $I$ (that is any element of the form $c_1\overrightarrow{r}+c_2\overrightarrow{g}$ with $c_1+c_2=n-2$), so $I$ must be an order ideal of the subposet consisting of all elements of $P_n$ not greater than or equal to $a$ (that is, elements of the form $c_1\overrightarrow{r}+c_2\overrightarrow{g}$ with $c_1+c_2<n-2$). This subposet is $P_{n-1}$. So there are $2^{n-2}$ order ideals of $P_n$ not containing $a$. Thus there are $2\cdot 2^{n-2} = 2^{n-1}$ order ideals of $P_n$ and the rank generating function is 
\[F(P_n,q)=(1+q^{n-1})\prod_{j=1}^{n-2} (1+q^j) = \prod_{j=1}^{n-1} (1+q^j).\]
\end{proof}

\begin{figure}[htbp]
\centering
\includegraphics[scale=0.64]{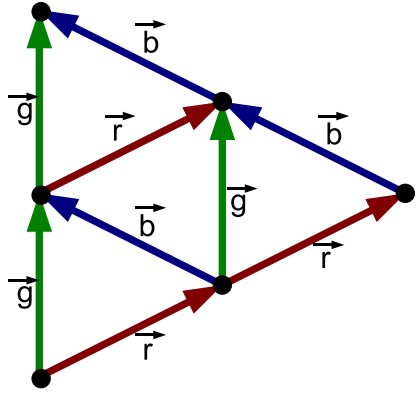}
\caption{The Hasse diagram of $P_4$ with each vector labeled}
\label{fig:p4}
\end{figure}

Let $P_n(S)$ where $S\subseteq \{r,b,g\}$ denote the poset with all the same elements as $P_n$, but with certain edges removed so that the Hasse diagram edges of $P_n(S)$ are only the vectors whose corresponding letters are in $S$. Note that $P_n(\{r,b\})$ has the same number of order ideals as $P_n$ since whenever two poset elements are connected by the vector $\overrightarrow{g}$ there is also a path from one to the other consisting of the vectors $\overrightarrow{r}$ and $\overrightarrow{b}$. From our construction we know that the other two posets $P_n(S)$ with $|S|=2$, $P_n(\{b,g\})$ and $P_n(\{r,g\})$, are dual posets (see Figure~\ref{fig:p4}).We now count the order ideals of the isomorphic posets $P_n(\{b,g\})$ and $P_n^*(\{r,g\})$, obtaining a famous number as our answer.

\begin{theorem}
\label{lemma:cat}
\[|J(P_n(\{b,g\}))|=|J(P_n(\{r,g\}))|=C_n=\frac{1}{n+1}{2n \choose n}.\] 
\[F(J(P_n(\{b,g\})),q)=F(J(P_n^*(\{r,g\})))=C_n(q)\]
where $C_n(q)$ is the Carlitz-Riordan $q$-Catalan number defined by the recurrence
\begin{equation}
\label{eq:carlitz}
C_n(q)=\displaystyle\sum_{k=1}^n q^{k-1} C_{k-1}(q) C_{n-k}(q) 
\end{equation}
with initial conditions $C_0(q)=C_1(q)=1$.
\end{theorem}

We give a proof below, but also note this is a well-known fact since $P_n(r,g)$ is also known as the root poset of type $A_{n-1}$, whose order ideals are counted by the Catalan number $C_n$.

\begin{proof}
We prove this theorem by constructing a bijection between order ideals of $P_n(\{b,g\})$ and Dyck paths of $2n$ steps. A Dyck path is a lattice path in the plane from $(0,0)$ to $(2n,0)$ with steps $(1,1)$ and $(1,-1)$ which never goes below the $x$-axis. Dyck paths from $(0,0)$ to $(2n,0)$ are counted by the Catalan number $C_n$~\cite{STANLEY}. To construct our bijection we will rotate the axes of our Dyck path slightly as in Figure~\ref{fig:dyck}.  The Carlitz-Riordan $q$-Catalan numbers $C_n(q)$ weight Dyck paths by $q$ to the power of the number of complete unit squares under the path (or rhombi in our rotated picture)~\cite{CARLITZRIORDAN}. Our claim is that Dyck paths with this weight are in bijection with order ideals of the posets $P_n(\{b,g\})$ weighted by number of elements in the order ideal. The bijection proceeds by overlaying the Dyck path on $P_n(\{b,g\})$ as in Figure~\ref{fig:dyck} and circling every poset element which is strictly below the path. Each circled element is the southeast corner of a unit rhombus under the Dyck path. Thus the order ideal consisting of all the circled elements corresponds to the Dyck path and the weight is preserved. Therefore $|J(P_n(\{b,g\}))|=C_n$ and the rank generating function $F(J(P_n(\{b,g\})),q)$ equals $C_n(q)$. Then since $P_n(\{b,g\})\simeq P_n^*(\{r,g\})$ the result follows by poset isomorphism.
\end{proof}

\begin{figure}[htb]
\centering
\includegraphics[scale=0.52]{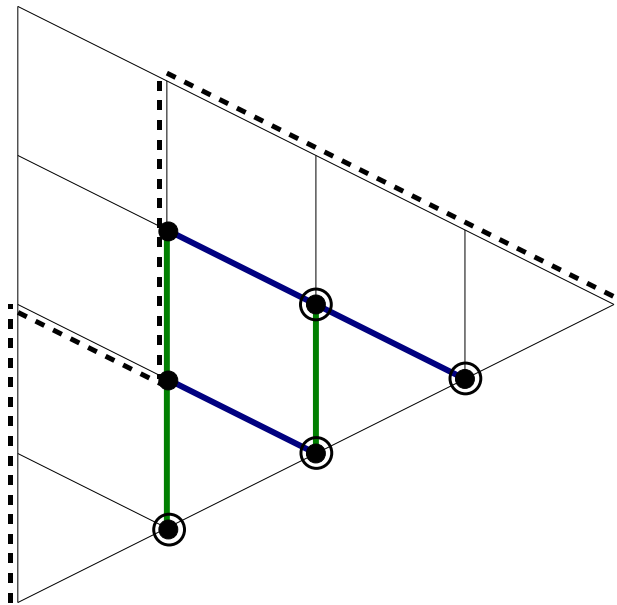}
\caption[A Dyck path overlaid on $P_4(\{b,g\})$]{A Dyck path (dashed line) overlaid on $P_4(\{b,g\})$. The order ideal corresponding to this path consists of all the circled poset elements.}
\label{fig:dyck}
\end{figure}

We now construct the tetrahedral poset $T_n$ as a three-dimensional analogue of the poset $P_n$. 
Define the unit vectors $\overrightarrow{y}=(\frac{\sqrt{3}}{6},\frac{1}{2},\frac{\sqrt{6}}{3})$, $\overrightarrow{o}=(\frac{-\sqrt{3}}{3},0,\frac{\sqrt{6}}{3})$, $\overrightarrow{s}=(\frac{-\sqrt{3}}{6},\frac{1}{2},-\frac{\sqrt{6}}{3})$. 
We use these vectors along with the previously defined vectors $\overrightarrow{r}$, $\overrightarrow{g}$, and $\overrightarrow{b}$ to define $T_n$ by drawing its Hasse diagram. First let the elements of $T_n$ be defined as the coordinates of all the points reached by linear combinations of $\overrightarrow{r}$, $\overrightarrow{g}$, and $\overrightarrow{y}$. Thus as a set $T_n =\{c_1 \overrightarrow{r} +c_2 \overrightarrow{g}+c_3\overrightarrow{y}, \mbox{ } c_1,c_2,c_3\in \mathbb{Z}_{\ge 0}, \mbox{ } c_1+c_2+c_3\le n-2\}$. To obtain the partial order on $T_n$ let all the vectors $\overrightarrow{r}$, $\overrightarrow{g}$, and $\overrightarrow{y}$ used to define the elements of $T_n$ be directed edges in the Hasse diagram, and additionally draw into the Hasse diagram as directed edges the vectors $\overrightarrow{b}$, $\overrightarrow{o}$, and $\overrightarrow{s}$ between poset elements wherever possible. See Figure~\ref{fig:tetravec} for an example of this construction.

\begin{figure}[htbp]
\centering
\includegraphics[scale=.9]{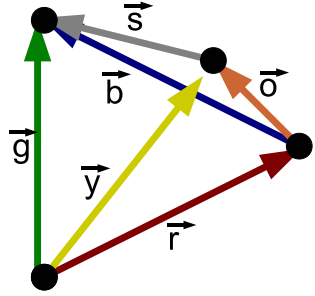}
\caption{The Hasse diagram of $T_3$ with each vector labeled}
\label{fig:tetravec}
\end{figure}

This construction yields a poset whose Hasse diagram is in the shape of a tetrahedron. We now think of each edge in the Hasse diagram as having a color corresponding to the first letter of the unit vector as follows: $\overrightarrow{r}$ equals red, $\overrightarrow{o}$ equals orange, $\overrightarrow{y}$ equals yellow, $\overrightarrow{g}$ equals green, $\overrightarrow{b}$ equals blue, and $\overrightarrow{s}$ equals silver (see Figure~\ref{fig:tet}).
The partial order of $T_n$ is defined so that the corner vertex with edges colored red, green, and yellow is the smallest element, the corner vertex with edges colored silver, green, and blue is the largest element, and the other two corner vertices are ordered such that the one with silver, yellow, and orange edges is above the one with orange, red, and blue edges.
$T_n$ restricted to only the red, green, and blue edges is isomorphic to the disjoint sum of $P_j$ for $2\le j\le n$. Thus $T_n$ can be thought of as the poset which results from beginning with the poset $P_n$, overlaying the posets $P_{n-1},P_{n-2},\ldots,P_3,P_2$ successively, and connecting each $P_{i}$ to $P_{i-1}$ in a certain way by the orange, yellow, and silver edges.

\begin{figure}[htbp]
\centering
\includegraphics[scale=0.55]{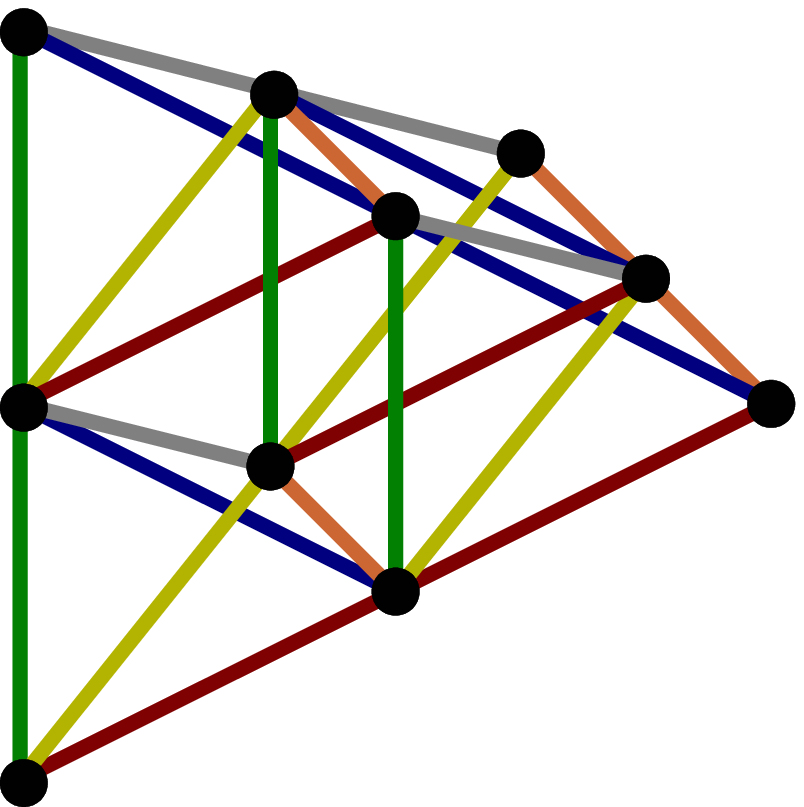}
\caption{The Hasse diagram of $T_4$}
\label{fig:tet}
\end{figure}

We now examine the number of order ideals of the elements of $T_n$ when we drop the ordering relations corresponding to edges of certain colors in the Hasse diagram. There are $2^6 = 64$ possible combinations of the six colors, but some of the colors are induced by the combination of others, so the number of distinct possibilities to consider is reduced. 
We wish to consider only subsets of the colors which include all the induced edges. Thus whenever we include red and blue we must also include green. Similarly, orange and silver induce blue, silver and yellow induce green, and red and orange induce yellow. We summarize this in the following definition.
\begin{definition}
Let a subset $S$ of the six colors $\{$red, blue, green, orange, yellow, silver$\}$ (abbreviated $\{r,b,g,o,y,s\}$) be called \emph{admissible} if all of the following are true:
\begin{itemize}
\item If $\{r,b\}\subseteq S$ then $g\in S$ 
\item If $\{o,s\}\subseteq S$ then $b\in S$ 
\item If $\{s,y\}\subseteq S$ then $g\in S$ 
\item If $\{r,o\}\subseteq S$ then $y\in S$
\end{itemize}
\end{definition}

This reduces the 64 possibilities to 40 admissible sets of colors to investigate.
Surprisingly, for almost all of these admissible sets of colored edges, there exists a nice product formula for the number of order ideals of $T_n$ restricted to these edges and a bijection between these order ideals and an interesting set of combinatorial objects.
 
Our notational convention will be such that given an admissible subset $S$ of the colors $\{r,b,g,o,y,s\}$, $T_n(S)$ denotes the poset formed by the vertices of $T_n$ together with all the edges whose colors are in $S$. The induced colors will be in parentheses.
We summarize below (with more explanation to come in this and the following sections) the posets $T_n(S)$ and lattices of order ideals $J(T_n(S))$ associated to each of the 40 sets of colors $S$, grouping them according to the number of colors in $S$. See Figure~\ref{fig:bigbigpic}.

We now state the product formulas for the number of order ideals of $T_n(S)$ for $S$ an admissible set of colors, along with the rank generating functions $F(P,q)$ wherever we have them. For the sake of comparison we have also written each formula as a product over the same indices $1\le i\le j\le k\le n-1$ in a way which is reminiscent of the MacMahon box formula. See Figures~\ref{fig:zeroone} through \ref{fig:tetra} for the Hasse diagram of a poset from each class for $n=4$.

\begin{theorem}
\[
F(J(T_n(\emptyset)),q)=(1+q)^{{n+1\choose 3}}=\displaystyle\prod_{1\le i\le j\le k\le n-1} \frac{[2]_q}{[1]_q}.
\]
\end{theorem}
\begin{proof}
$J(T_n(\emptyset))$ is simply the Boolean algebra of rank ${n+1 \choose 3}$ whose generating function is as above.
\end{proof}

\begin{figure}[hbtp]
\centering
$\begin{array}{lcccr}
\includegraphics[scale=0.45]{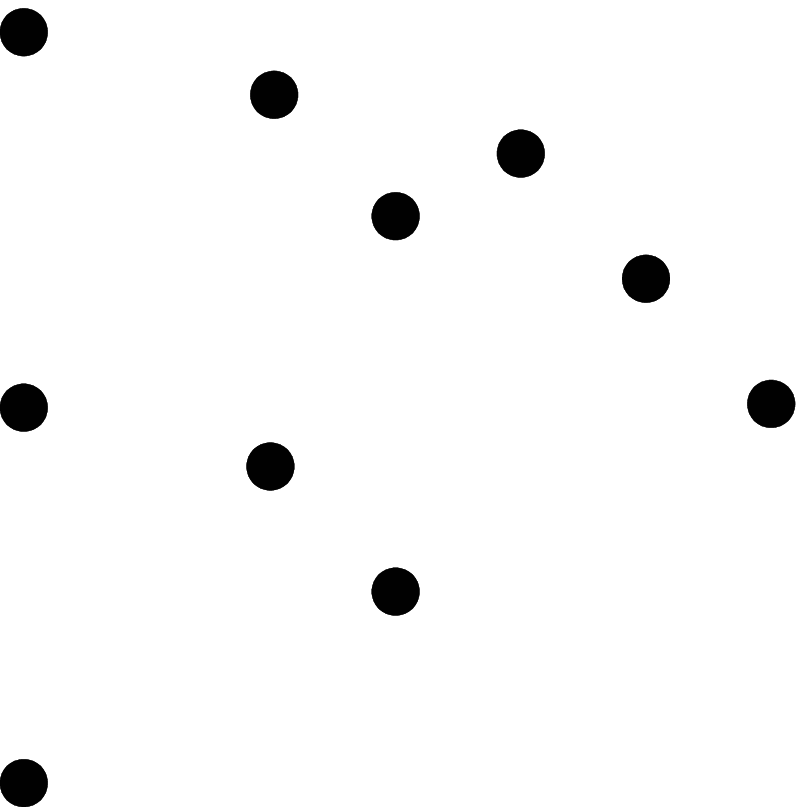}&&
&&
\includegraphics[scale=0.45]{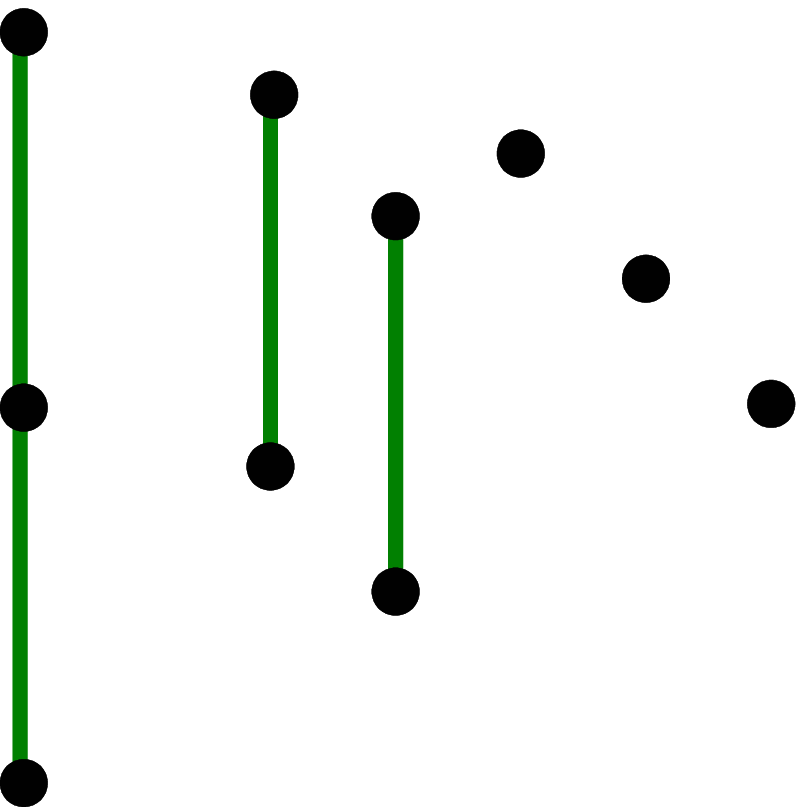}
\end{array}$
\caption{Left: $T_4(\emptyset)$ \hspace{.2cm} Right: $T_4(\{g\})$}
\label{fig:zeroone}
\end{figure}

\begin{theorem}
For any color $x\in \{r,b,y,g,o,s\}$
\[
F(J(T_n(\{x\})),q)=\prod_{j=1}^n j!_q=\displaystyle\prod_{1\le i\le j\le k\le n-1} \frac{[i+1]_q}{[i]_q}.
\]
\end{theorem}
\begin{proof}
$T_n(\{x\})$ is the disjoint sum of $n-j$ chains of length $j-1$ as $j$ goes from 1 to $n-1$. So the number of order ideals is the product of the number of order ideals of each chain, which can be expressed using factorials. The $q$-case is also clear.
\end{proof}

The proofs of the following two theorems about the two-color posets are in Section \ref{sec:two}.
\begin{theorem}
\label{thm:twoopp}
If $S\in \{\{g,o\},\{r,s\},\{b,y\}\}$ then
\[
F(J(T_n(S)),q)=\displaystyle\prod_{j=1}^n {n\brack j}_q=\displaystyle\prod_{1\le i\le j\le k\le n-1} \frac{[j+1]_q}{[j]_q}.
\]
\end{theorem}

\begin{theorem}
\label{thm:twoadj}
If $S_1\in \{\{b,g\},\{b,s\},\{y,o\},\{g,s\}\}$ and $S_2\in \{\{r,y\},\{r,g\},\{y,g\},$ $\{b,o\}\}$ 
then
\[
|J(T_n(S_1))|=|J(T_n(S_2))|=\displaystyle\prod_{j=1}^n C_j =\displaystyle\prod_{j=1}^n \frac{1}{j+1} {2j \choose j}=\displaystyle\prod_{1\le i\le j\le k\le n-1} \frac{i+j+2}{i+j},
\]
where $C_j$ is the $j$th Catalan number. Also,
\[
F(J(T_n(S_1)),q)=F(J(T_n^*(S_2)),q)=\displaystyle\prod_{j=1}^n C_j(q),
\]
where the $C_j(q)$ are the $j$th Carlitz-Riordan $q$-Catalan numbers defined in Equation~(\ref{eq:carlitz}).
\end{theorem}

\begin{figure}[htbp]
\centering
$\begin{array}{lcccr}
\includegraphics[scale=0.45]{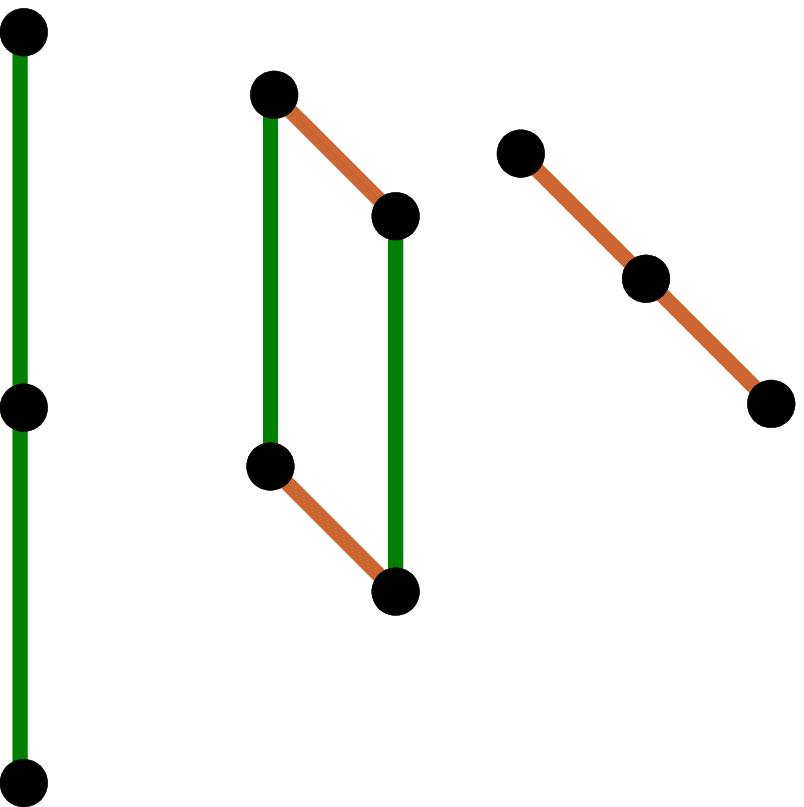}&&
&&
\includegraphics[scale=0.45]{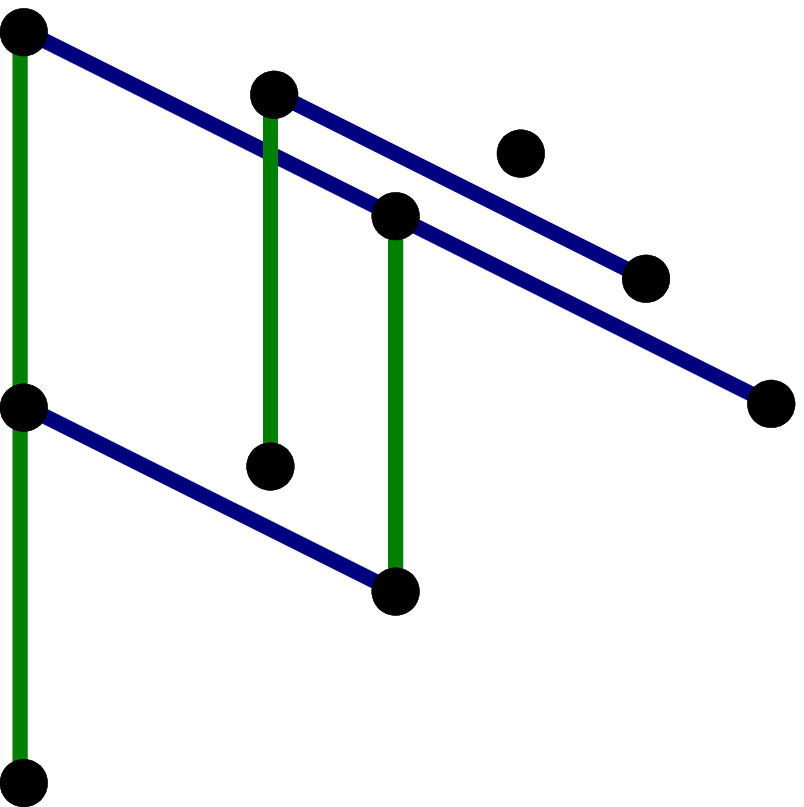}
\end{array}$
\caption{Left: Binomial poset $T_4(\{o,g\})$ \hspace{.2cm} Right: Catalan poset $T_4(\{b,g\})$}
\label{fig:two}
\end{figure}

The proof of the next theorem on the three-color posets and the relationship between these order ideals and SSYT and tournaments is the subject of Section~\ref{sec:three}.
\begin{theorem}
\label{thm:3nadj}
If $S$ is an admissible subset of $\{r,b,g,o,y,s\}$, $|S|=3$, and $S\notin \{\{r,g,y\},\{s,b,r\}\}$ then
\[
F(J(T_n(S)),q)
=\displaystyle\prod_{j=1}^{n-1} (1+q^j)^{n-j}=\displaystyle\prod_{1\le i\le j\le k\le n-1} \frac{[i+j]_q}{[i+j-1]_q}.
\]
Thus if we set $q=1$ we have $|J(T_n(S))|= 2^{{n\choose 2}}$.
\end{theorem}

\begin{figure}[htbp]
\centering
$\begin{array}{lcccr}
\includegraphics[scale=0.45]{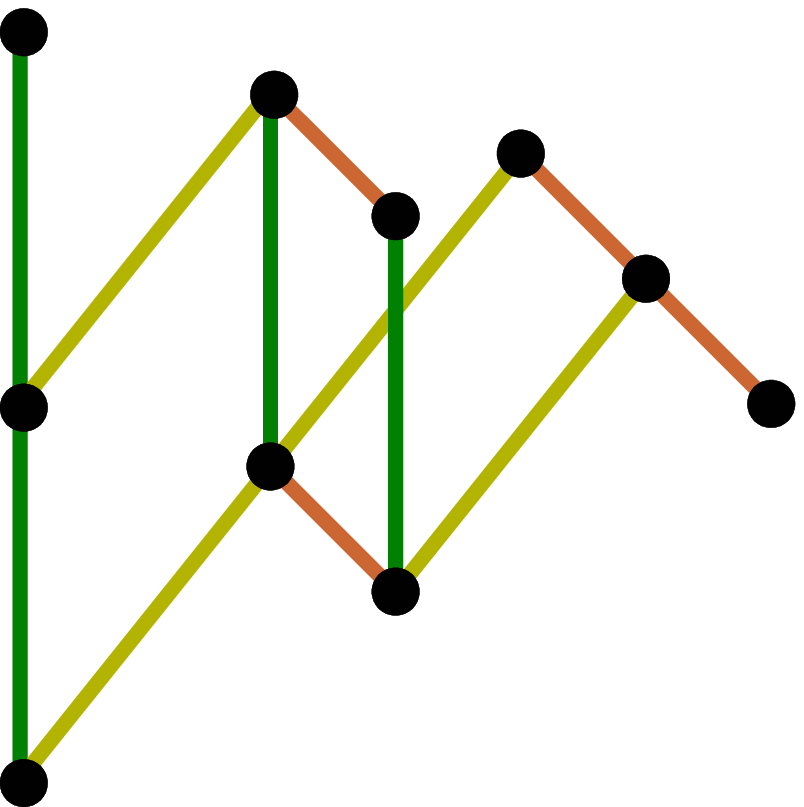}&&
&&
\includegraphics[scale=0.45]{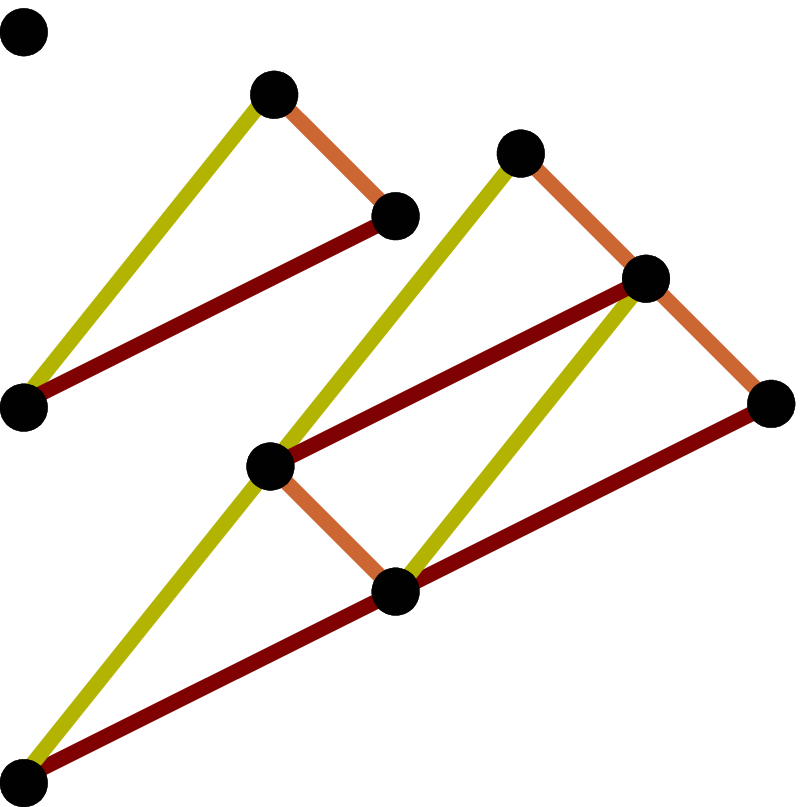}
\end{array}$
\caption{Left: SSYT poset $T_4(\{o,g,y\})$ \hspace{.2cm} Right: Tournament poset $T_4(\{r,(y),o\})$}
\label{fig:threenice}
\end{figure}

There seems to be no nice product formula for the number of order ideals of the dual posets $T_n(\{r,g,y\})$ and $T_n(\{s,b,r\})$ which are the only two ways to pick three adjacent colors while not inducing any other colors. We have calculated the number of order ideals for $n=1$ to $6$ to be: 1, 2, 9, 96, 2498, 161422.

The proof of the following theorem on the four-color posets and the relationship between these order ideals and ASMs and TSSCPPs is the subject of Section~\ref{sec:four}.
\begin{theorem}
\label{thm:4four}
If $S$ is an admissible subset of $\{r,b,g,o,y,s\}$ and $|S|=4$ then 
\[
|J(T_n(S))|=\displaystyle\prod_{j=0}^{n-1} \frac{(3j+1)!}{(n+j)!}=\displaystyle\prod_{1\le i\le j\le k\le n-1} \frac{i+j+k+1}{i+j+k-1}.
\]
\end{theorem}

\begin{figure}[htbp]
\centering
$\begin{array}{lcccr}
\includegraphics[scale=0.45]{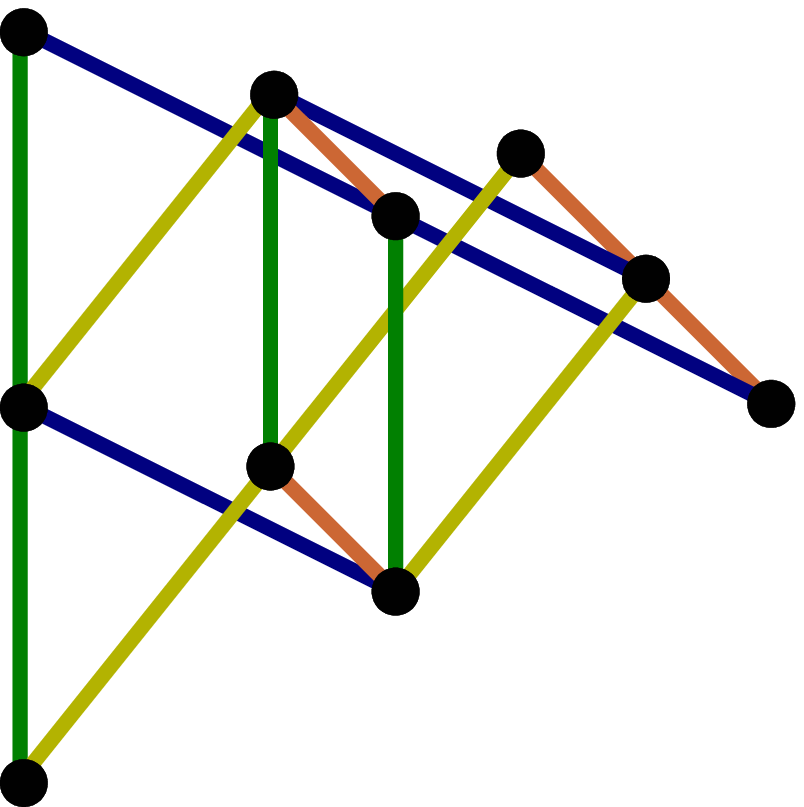}&&
&&
\includegraphics[scale=0.45]{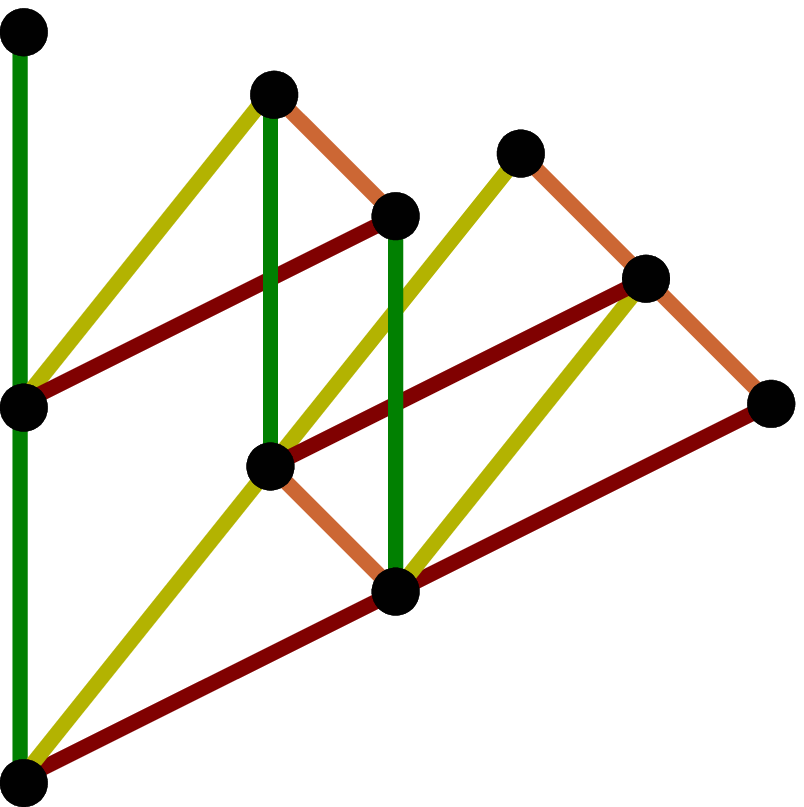}
\end{array}$
\caption[Left: ASM $T_4(\{o,g,y,b\})$ \hspace{.2cm} Right: TSSCPP $T_4(\{r,(y),o,g\})$]{Left: ASM poset $T_4(\{o,g,y,b\})$ \hspace{.2cm} Right: TSSCPP poset $T_4(\{r,(y),o,g\})$}
\label{fig:fournice}
\end{figure}

There are two different cases for five colors: one case consists of the dual posets $T_n(\{(g),(b),o,y,s\})$ and $T_n(\{r,b,(g),o,(y)\})$ and the other case is $T_n(\{r,b,s,(y),g\})$. A nice product formula has not yet been found for either case. 
See further discussion of these posets in Section~\ref{sec:five}.

\begin{figure}[htbp]
\centering
$\begin{array}{lcccr}
\includegraphics[scale=0.45]{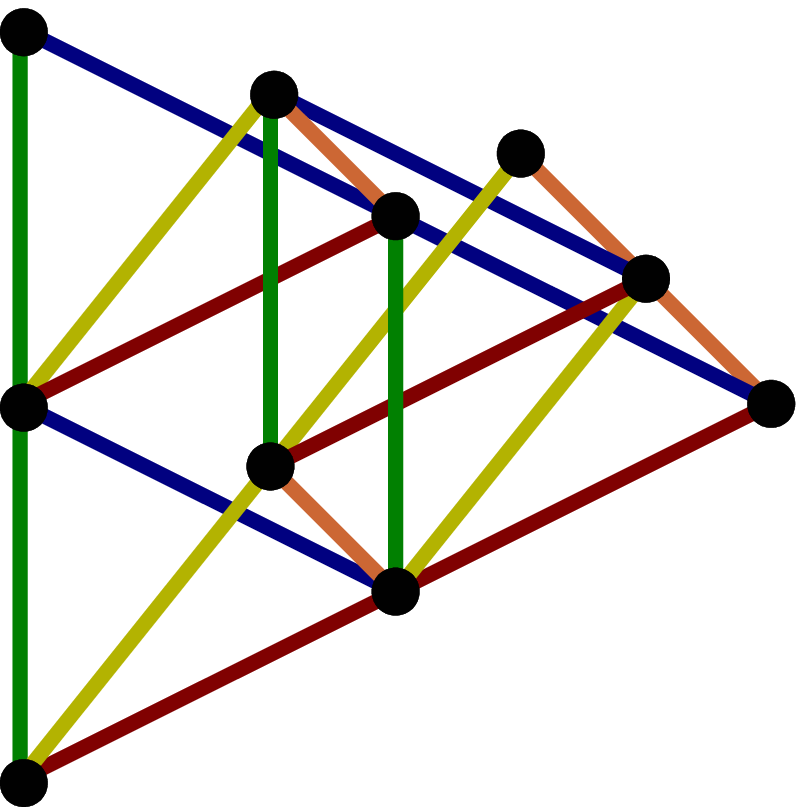}&&
&&
\includegraphics[scale=0.45]{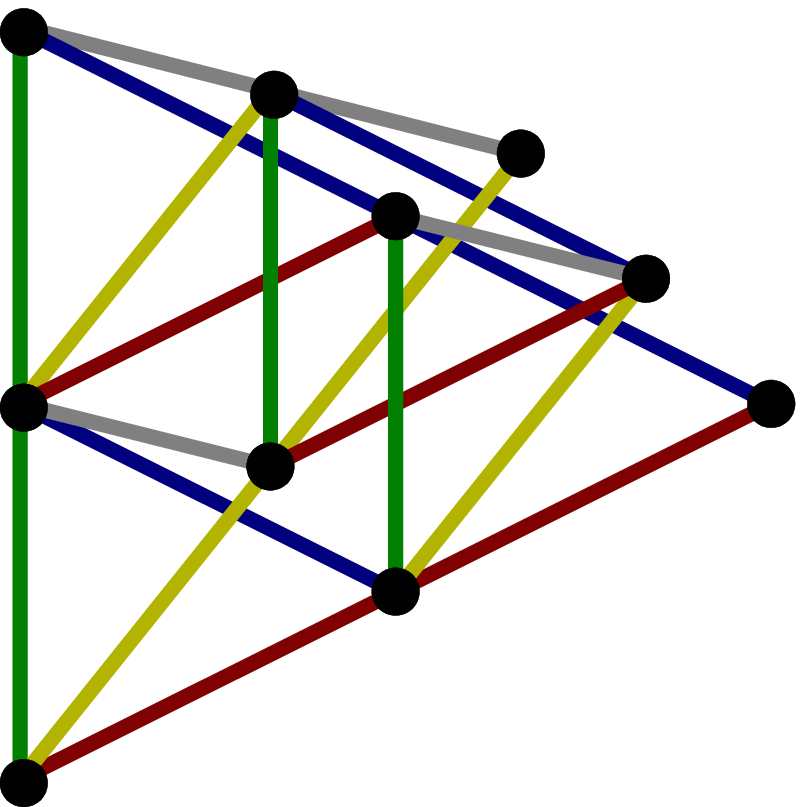}
\end{array}$
\caption{Left: ASM $\cap$ TSSCPP $T_4(\{o,(g),(y),b,r\})$ \hspace{.2cm} Right: TSSCPP $\cap$ TSSCPP $T_4(\{g,(y),b,r,s\})$}
\label{fig:five}
\end{figure}

Finally, the order ideals of the full tetrahedron $T_n$ are counted by the following theorem, which is proved in Section~\ref{sec:six} via a bijection with TSPPs inside an $(n-1)\times (n-1)\times (n-1)$ box.
\begin{theorem}
\label{thm:6six}
\[
|J(T_n)|
=\displaystyle\prod_{1\le i\le j\le n-1} \frac{i+j+n-2}{i+2j-2}=\displaystyle\prod_{1\le i\le j\le k\le n-1} \frac{i+j+k-1}{i+j+k-2}.
\]
\end{theorem}

\begin{figure}[htbp]
\centering
\includegraphics[scale=0.45]{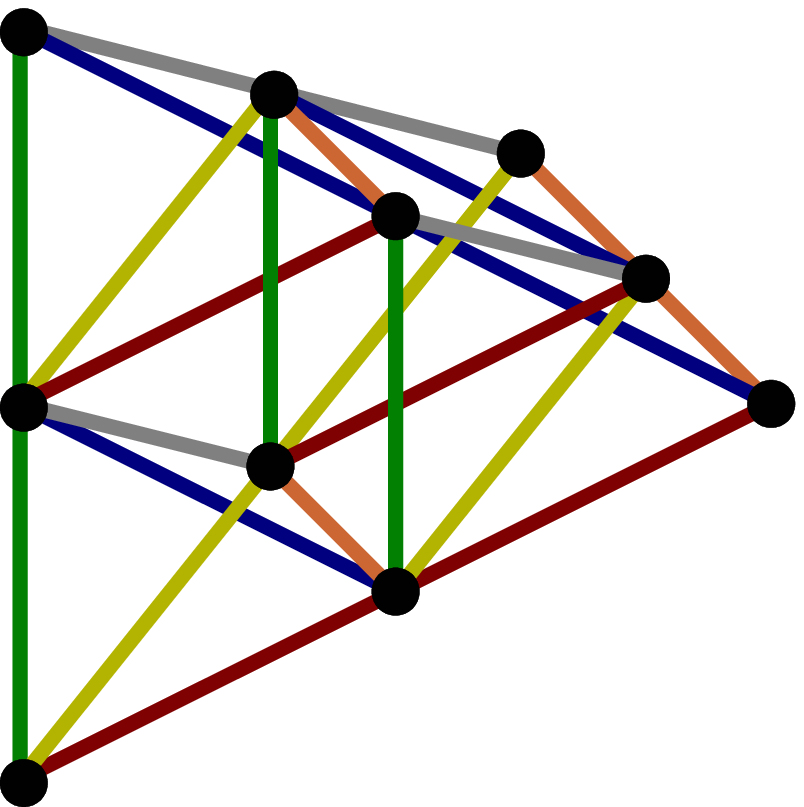}
\caption{TSPP poset $T_4$}
\label{fig:tetra}
\end{figure}

\begin{figure}
\centering
\includegraphics[scale=0.48]{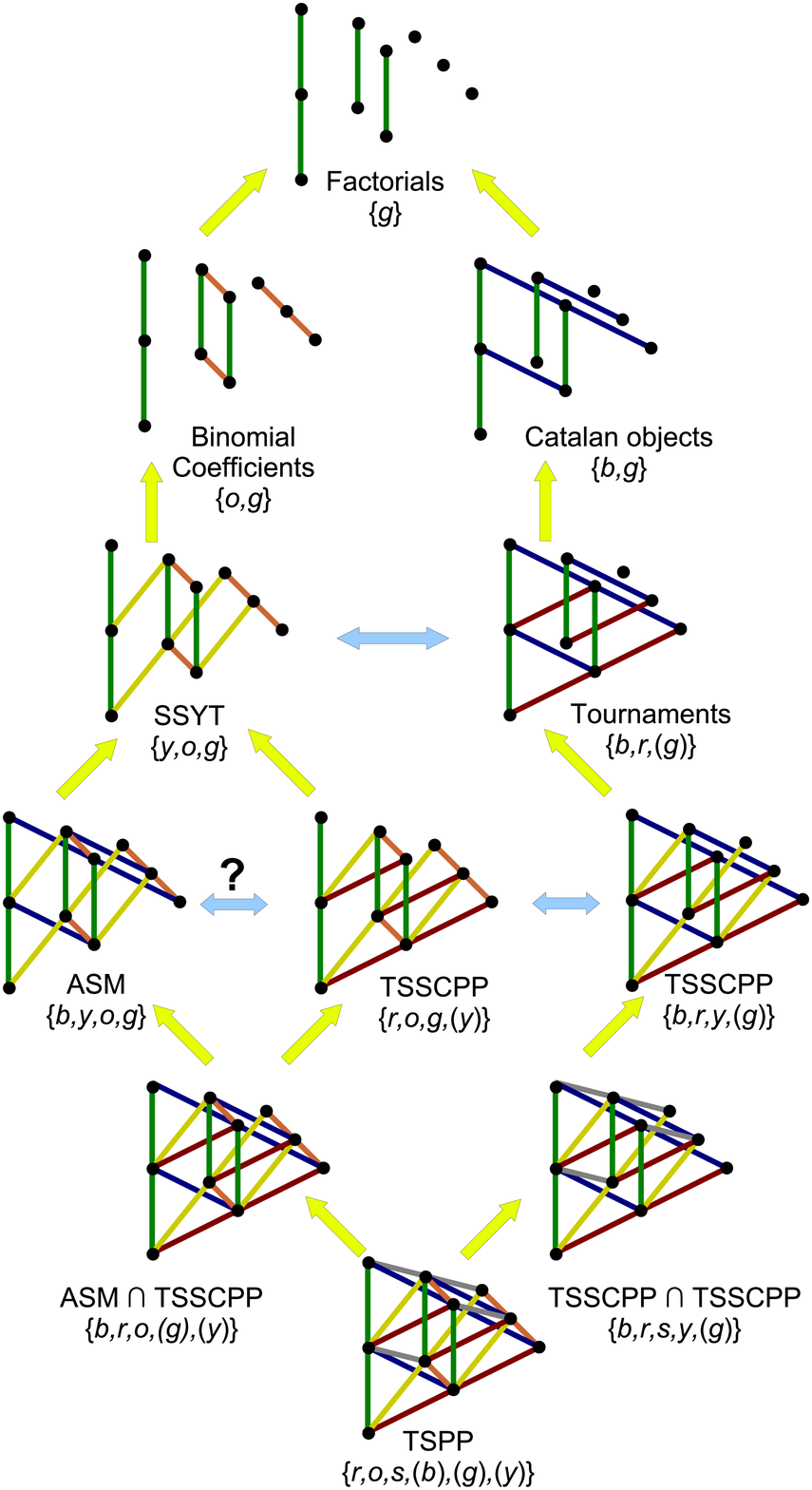}
\caption[The big picture of inclusions and bijections between order ideals]{The big picture of inclusions and bijections between order ideals $J(T_n(S))$. The bijection between three color posets is in Section~\ref{sec:three} and the bijections between TSSCPP posets is by poset isomorphism. The only missing bijection between sets of the same size is between ASM and TSSCPP.}
\label{fig:bigbigpic}
\end{figure}

Next we give bijections between order ideals of $T_n(S)$, $S$ an admissible subset of $\{r,b,g,y,o,s\}$, and arrays of integers with certain inequality conditions. These bijections will be used frequently throughout the rest of the paper to prove the theorems stated above and other facts about $T_n(S)$ and the related combinatorial objects.

\begin{definition}
\label{def:aofs}
Let $S$ be an admissible subset of $\{r,b,g,y,o,s\}$ and suppose $g\in S$. Define $X_n(S)$ to be the set of all integer arrays $x$ of staircase shape $\delta_n=(n-1)~(n-2)\ldots 3~2~1$ with entries $x_{i,j}$ for $1\le i\le n-1$ and $1\le j \le n-i$ which satisfy both
$0\le x_{i,j}\le j$ and the following inequality conditions corresponding to the additional colors in $S$:
\begin{itemize}
\item orange: $x_{i,j}\le x_{i+1,j}$
\item red: $x_{i,j}\le x_{i-1,j+1}$
\item yellow: $x_{i,j}\le x_{i,j+1}$
\item blue: $x_{i,j}\le x_{i+1,j-1}+1$
\item silver: $x_{i,j}\le x_{i,j-1}+1$
\end{itemize}
\end{definition}

\begin{proposition}
\label{prop:arrays}
If $S$ is an admissible subset of $\{r,b,g,y,o,s\}$ and $g\in S$ then $X_n(S)$ is in weight-preserving bijection with the set of order ideals $J(T_n(S))$ where the weight of $x\in X_n(S)$ equals the sum of the entries of $x$ and the weight of $I\in J(T_n(S))$ equals $|I|$. 
\end{proposition}
\begin{proof}
Let $S$ be an admissible subset of $\{r,b,g,y,o,s\}$ and suppose $g\in S$. Recall that $T_n$ is made up of the layers $P_k$ where $2\le k\le n$ and that $P_k$ contains $k-1$ green-edged chains of length $k-1,\ldots,2,1$. For each $P_k\subseteq T_n$ let the $k-1$ green chains inside $P_k$ determine the entries $x_{i,j}$ of an integer array on the diagonal where $i+j=k$. In particular, given an order ideal $I$ of $T_n(S)$ form an array $x$ by setting $x_{i,j}$ equal to the number of elements in the induced order ideal of the green chain of length $j$ inside $P_{i+j}$. This defines $x$ as an integer array of staircase shape $\delta_n$ whose entries satisfy $0\le x_{i,j}\le j$. Also since each entry $x_{i,j}$ is given by an induced order ideal and since each element of $T_n$ is in exactly one green chain we know that $|I|=\sum_{i,j} x_{i,j}$. Thus the weight is preserved.

Now it is left to determine what the other colors mean in terms of the array entries. Since the colors red and blue connect green chains from the same $P_k$ we see that inequalities corresponding to red and blue should relate entries of $x$ on the same northeast to southwest diagonal of $x$. So if $r\in S$ then $x_{i,j}\le x_{i-1,j+1}$ and if $b\in S$ then $x_{i,j}\le x_{i+1,j-1}+1$. The colors yellow, orange, and silver connect $P_k$ to $P_{k+1}$ for $2\le k\le n-1$. So from our construction we see that 
if $o\in S$ then $x_{i,j}\le x_{i+1,j}$, if $y\in S$ then $x_{i,j}\le x_{i,j+1}$, and if $s\in S$ then $x_{i,j}\le x_{i,j-1}+1$.
\end{proof}

A useful transformation of the elements in $X_n(S)$ is given in the following definition.
\begin{definition}
\label{def:bofs}
Let $S$ be an admissible subset of $\{r,b,g,y,o,s\}$ and suppose $g\in S$. Define $Y_n(S)$ to be the set of all integer arrays $y$ of staircase shape $\delta_n=(n-1)~(n-2)\ldots 3~2~1$ whose entries $y_{i,j}$ satisfy both
$i\le y_{i,j}\le j+i$ and the following inequality conditions corresponding to the additional colors in $S$:
\begin{itemize}
\item orange: $y_{i,j} < y_{i+1,j}$
\item red: $y_{i,j}\le y_{i-1,j+1} +1$
\item yellow: $y_{i,j}\le y_{i,j+1}$
\item blue: $y_{i,j}\le y_{i+1,j-1}$
\item silver: $y_{i,j}\le y_{i,j-1}+1$
\end{itemize}
\end{definition}

\begin{proposition}
\label{prop:yofs}
If $S$ is an admissible subset of $\{r,b,g,y,o,s\}$ and $g\in S$ then
$Y_n(S)$ is in weight-preserving bijection with $J(T_n(S))$ where the weight of $y\in Y_n(S)$ is given by $\sum_{i=1}^{n-1} \sum_{j=1}^{n-i} (y_{i,j} - i)$ and the weight of $I\in J(T_n(S))$ equals $|I|$. 
\end{proposition}
\begin{proof}
Let $S$ be an admissible subset of $\{r,b,g,y,o,s\}$ and suppose $g\in S$. For $x\in X_n(S)$ consider the array $y$ with entries $y_{i,j}= x_{i,j}+i$. 
The transformation $y_{i,j}= x_{i,j}+i$ affects the meaning of the colors in terms of inequalities. Any inequality which compares elements in two different rows is altered. Thus orange now corresponds to $y_{i,j} < y_{i+1,j}$, red now corresponds to $y_{i,j}\le y_{i-1,j+1} +1$, blue now corresponds to $y_{i,j}\le y_{i+1,j-1}$, and the other inequalities are unaffected. Also since $y_{i,j}= x_{i,j}+i$ it follows that the weight of $y$ is given by $\sum_{i,j} (y_{i,j} - i)$. Thus $Y_n(S)$ is in weight-preserving bijection with $X_n(S)$ so that $Y_n(S)$ is in weight-preserving bijection with $J(T_n(S))$.
\end{proof}

We make one final transformation of these arrays which will be needed later. It will sometimes be helpful to think of the arrays $Y_n(S)$ as having a 0th column with entries $1~2~3\ldots n$. 
\begin{definition}
\label{def:yplus}
Let $S$ be an admissible subset of $\{r,b,g,y,o,s\}$ and suppose $g\in S$. Define $Y_n^+(S)$ to be the set of all arrays $y\in Y_n(S)$ with an additional column (call it column 0) with entries $y_{i,0}=i$ for $1\le i\le n$. 
\end{definition}
Thus $y\in Y_n^+(S)$ is an array of shape $n~(n-1)\ldots 3~2~1$ with fixed 0th column $1~2~3\ldots n$. 

\begin{proposition}
\label{prop:yplusbij}
If $S$ is an admissible subset of $\{r,b,g,y,o,s\}$ and $g\in S$ then
$Y_n^+(S)$ is in weight-preserving bijection with $J(T_n(S))$ where the weight of $y\in Y_n^+(S)$ is given by $\displaystyle\sum_{i=1}^{n-1} \displaystyle\sum_{j=0}^{n-i} (y_{i,j} - i)$ and the weight of $I\in J(T_n(S))$ equals $|I|$. 
\end{proposition}
\begin{proof}
$Y_n^+(S)$ is trivially in bijection with $Y_n(S)$ by the removal of the 0th column. The weight is preserved since $y_{i,0}=i$ so that $\sum_{i=1}^{n-1} (y_{i,0}-i)=0$. Therefore the added column entries contribute nothing to the weight.
\end{proof}

\section{The two color posets---Catalan numbers and binomial coefficients}
\label{sec:two}

In this section we prove Theorems~\ref{thm:twoopp} and \ref{thm:twoadj} about the weighted enumeration of $J(T_n(S))$ when $|S|=2$. One theorem gives a product of Catalan numbers and the other a product of binomial coefficients.  
We begin by proving the enumeration of the two-color posets $J(T_n(S))$ when $S\in \{\{g,o\},\{r,s\},\{b,y\}\}$.
\begin{thm216}
For $S\in \{\{g,o\},\{r,s\},\{b,y\}\}$
\begin{equation}
F(J(T_n(S)),q)=\displaystyle\prod_{j=1}^n {n\brack j}_q=\displaystyle\prod_{1\le i\le j\le k\le n-1} \frac{[j+1]_q}{[j]_q}.
\end{equation}
\end{thm216}
\begin{proof}
Consider the staircase shape arrays $Y_n(\{g,o\})$ which strictly decrease down columns and have no conditions on the rows. Thus in a column of length $j$ there must be $j$ distinct integers between $1$ and $n$; this is counted by ${n\choose j}$. If we give a weight to each of these integers of $q$ to the power of that integer minus its row, we have a set $q$-enumerated by the $q$-binomial coefficient ${n\brack j}_q$. Thus $\prod_{j=1}^n {n\brack j}_q$ is the generating function of the arrays $Y_n(\{g,o\})$.
Recall from Proposition~\ref{prop:yofs} that $Y_n(\{g,o\})$ is in weight-preserving bijection with $J(T_n(\{g,o\}))$ where the weight of $y\in Y_n(\{g,o\})$ is given by $\sum_{i,j} (y_{i,j} - i)$ and the weight of $I\in J(T_n(\{g,o\}))$ equals $|I|$. Thus the rank generating function of the lattice of order ideals $F(J(T_n(\{g,o\})),q)$ equals $\prod_{j=1}^n {n\brack j}_q$.
The posets $T_n(\{r,s\})$ and $T_n(\{b,y\})$ are both isomorphic to $T_n(\{g,o\})$. Therefore the result follows by poset isomorphism (see~\cite{STRIKER_THESIS} for the explicit poset isomorphisms).
\end{proof}

Now suppose $S$ is admissible and $|S|=2$ but $S\notin \{\{g,o\},\{r,s\},\{b,y\}\}$. Then Theorem~\ref{thm:twoadj} applies. 
\begin{thm217}
If $S_1\in \{\{b,g\}$, $\{b,s\}$, $\{y,o\}$, $\{g,s\}\}$ and $S_2\in \{\{r,y\}$, $\{r,g\}$, $\{y,g\}$, $\{b,o\}\}$ 
then
\begin{equation}
\label{eq:2adj}
|J(T_n(S_1))|=|J(T_n(S_2)|=\displaystyle\prod_{j=1}^n C_j =\displaystyle\prod_{j=1}^n \frac{1}{j+1} {2j \choose j}=\displaystyle\prod_{1\le i\le j\le k\le n-1} \frac{i+j+2}{i+j},
\end{equation} 
where $C_j$ is the $j$th Catalan number. Also,
\begin{equation}
F(J(T_n(S_1)),q)=F(J(T_n^*(S_2)),q)=\displaystyle\prod_{j=1}^n C_j(q),
\end{equation}
where the $C_j(q)$ are the Carlitz-Riordan $q$-Catalan numbers.
\end{thm217}
\begin{proof}
$T_n(\{b,g\})$ is isomorphic to the disjoint sum of $P_j(\{b,g\})$ for $2\le j\le n$. Recall from Theorem~\ref{lemma:cat} that $|J(P_j(\{b,g\}))|=C_j=\frac{1}{j+1}{2j \choose j}$ and the rank generating function $F(J(P_j(\{b,g\})))$ equals $C_j(q)$ where $C_j(q)$ is the Carlitz-Riordan $q$-Catalan number defined by the recurrence
$C_j(q)=\sum_{k=1}^j q^{k-1} C_{k-1}(q) C_{j-k}(q)$
with initial conditions $C_0(q)=C_1(q)=1$.
Thus the number of order ideals $|J(T_n(\{b,g\}))|$ equals the product $\prod_{j=2}^n C_j$ and the rank generating function $F(J(T_n(\{b,g\})),q)$ equals the product $\prod_{j=2}^n C_j(q)$. Then since $C_1=C_1(q)=1$ we can write the product as $j$ goes from 1 to $n$. Finally, the posets $T_n(S_1)$ for any choice of $S_1\in \{\{b,g\}$, $\{b,s\}$, $\{y,o\}$, $\{g,s\}\}$ and the posets $T_n^*(S_2)$ for any $S_2\in \{\{r,y\}$, $\{r,g\}$, $\{y,g\}$, $\{b,o\}\}$ are all isomorphic, thus the result follows by poset isomorphism.
\end{proof}

Note that the rank generating function of this theorem is not the direct $q$-ification of the counting formula (\ref{eq:2adj}). This is because the rank generating function does not factor nicely so it cannot be written as a nice product formula, but instead as a recurrence. The direct $q$-ification of (\ref{eq:2adj}) is still the generating function of $J(T_n(\{b,g\}))$, but with a different weight. $\prod_{j=1}^n \frac{1}{[j+1]_q}{2j\brack j}_q$ is the product of the first $n$ MacMahon $q$-Catalan numbers. The MacMahon $q$-Catalan numbers $\frac{1}{[j+1]_q}{2j\brack j}_q$ count Dyck paths with $2j$ steps (or Catalan words of length $2j$) weighted by major index. The major index of a Dyck path is the sum of all $i$ such that step $i$ of the Dyck path is an increasing step and step $i+1$ is a decreasing step. 
Therefore $\prod_{j=1}^n \frac{1}{[j+1]_q}{2j\brack j}_q$ is the generating function for order ideals $I\in J(T_n(\{b,g\}))$ with weight not corresponding to $|I|$ but rather to the product of the major index of each Dyck path in bijection with the order ideal of $P_j$ induced by the elements of $I$ contained in $P_j$. 
See~\cite{QTCAT} for a nice exposition of both the MacMahon and Carlitz-Riordan $q$-Catalan numbers.

\section{Three color posets---tournaments and semistandard Young tableaux}
\label{sec:three}

In this section we prove Theorem~\ref{thm:3nadj} which states that the order ideals of the three color posets $J(T_n(S))$ where $S$ is admissible, $|S|=3$, and $S\notin \{\{r,g,y\},\{s,b,r\}\}$ are counted by the very simple expression $2^{n\choose 2}$. This is also the number of graphs on $n$ labeled vertices and equivalently the number of tournaments on $n$ vertices. A tournament is a complete directed graph with labeled vertices. 
We will first prove the weighted counting formula of Theorem~\ref{thm:3nadj} and then discuss bijections between the order ideals of the three color posets and tournaments. In the next section we will discuss the possibility of using the ideas from these bijections to find a bijection between ASMs and TSSCPPs.

We will separate Theorem~\ref{thm:3nadj} into two propositions since there are two nonisomorphic classes of posets $T_n(S)$ where $S$ is admissible, $|S|=3$, and $S\notin \{\{r,g,y\},\{s,b,r\}\}$. The first proposition is the case where $T_n(S)$ is a disjoint sum of posets and the second proposition is the case where $T_n(S)$ is a connected poset.
\begin{proposition}
\label{lemma:disjoint}
Suppose $S\in \{\{o,s,(b)\}, \{s,y,(g)\}, \{o,r,(y)\}, \{b,r,(g)\}\}$. Then
\[
F(J(T_n(S)),q)
=\displaystyle\prod_{j=1}^{n-1} (1+q^j)^{n-j}=\displaystyle\prod_{1\le i\le j\le k\le n-1} \frac{[i+j]_q}{[i+j-1]_q}.
\]
\end{proposition}
\begin{proof}
$T_n(\{b,r,(g)\})$ is a disjoint sum of the $P_j$ posets for $2\le j\le n$ whose order ideals are counted by $2^{j-1}$. Furthermore, by Theorem~\ref{lemma:pn} the rank generating function of $J(P_j)$ is given by $\prod_{i=1}^{j-1} (1+q^i)$ thus $F(T_n(\{b,r,(g)\}),q)$ is the product of $\prod_{i=1}^{j-1} (1+q^i)$ for $2\le j\le n$. Rewriting the product we obtain $\prod_{j=2}^n\prod_{i=1}^{j-1} (1+q^i)=\prod_{j=1}^{n-1} (1+q^j)^{n-j}$.
Now the posets $T_n(S)$ where $S\in \{\{o,s,(b)\}, \{s,y,(g)\}, \{o,r,(y)\}\}$ are isomorphic to $T_n(\{b,r,(g)\})$, so the result follows by poset isomorphism.
\end{proof}

Next we state the second proposition which proves Theorem~\ref{thm:3nadj} for the second case in which $T_n(S)$ is a connected poset. 
\begin{proposition}
\label{lemma:connected}
Suppose $S\in \{\{r,g,s\},\{o,b,y\},\{y,g,o\},\{b,g,o\},\{y,g,b\}\}$. Then
\[
F(J(T_n(S)),q)
=\displaystyle\prod_{j=1}^{n-1} (1+q^j)^{n-j}=\displaystyle\prod_{1\le i\le j\le k\le n-1} \frac{[i+j]_q}{[i+j-1]_q}.
\]
\end{proposition}



\begin{proof}
The arrays $Y_n(\{g,y,o\})$ are by Definition~\ref{def:bofs} semistandard Young tableaux (SSYT) of staircase shape $\delta_n$. Semistandard Young tableaux are integer arrays which are weakly increasing across rows and strictly increasing down columns. 
It is well known that the generating function of SSYT is given by the Schur function $s_\lambda$, thus $s_{\delta_n}(x_1,x_2,\ldots,x_n)$ is the generating function for $Y_n(\{g,y,o\})$. Now
\[
s_{\delta_n}(x_1,x_2,\ldots,x_n)=\frac{\det(x_i^{2(n-j)})_{i,j=1}^n}{\det(x_i^{n-j})_{i,j=1}^n}=\displaystyle\prod_{1\le i<j\le n}\frac{x_i^2-x_j^2}{x_i-x_j}=\displaystyle\prod_{1\le i<j\le n}(x_i+x_j)
\] 
using the algebraic Schur function definition and the Vandermonde determinant (see for example~\cite{MACDONALDSYM}). The principle specialization of this generating function yields the $q$-generating function $F(J(T_n(\{g,y,o\})),q)
=\prod_{j=1}^{n-1} (1+q^j)^{n-j}$.
The posets $T_n(S)$ where $S\in \{\{r,g,s\},\{o,b,y\},\{b,g,o\},\{y,g,b\}\}$ are isomorphic to $T_n(\{g,y,o\})$ so the result follows by poset isomorphism.
\end{proof}

Since $\{r,g,y\}$ and $\{s,b,r\}$ are the only admissible subsets not addressed by  Propositions~\ref{lemma:disjoint} and \ref{lemma:connected}, we have proved Theorem~\ref{thm:3nadj}.


We now discuss bijections between the order ideals of the three color posets and tournaments. First we give the bijection between tournaments on $n$ vertices and the order ideals of $T_n(\{b,r,(g)\})$. By poset isomorphism we can then extend this bijection to any of the disjoint three-color posets of Proposition~\ref{lemma:disjoint}. See Figure~\ref{fig:brgtourn} for the bijection when $n=3$.

For this bijection we will use the modified arrays $Y_n^+(S)$ from Definition~\ref{def:yplus}. Recall that $Y_n^+(S)$ are exactly the arrays $Y_n(S)$ with an additional column 0 added with fixed entries $1~2~3~\ldots~n$.
The colors blue and red correspond to inequalities on $Y_n^+(\{b,r,(g)\})$ such that as one goes up the southwest to northeast diagonals at each step the next entry has the choice between staying the same and decreasing by one. That is, for $\alpha\in Y_n^+(\{b,r,(g)\})$ and all $j\ge 1$, either $\alpha_{i,j} = \alpha_{i+1,j-1} \mbox{ or } \alpha_{i,j} = \alpha_{i+1,j-1} - 1$. Therefore since each of the ${n\choose 2}$ entries of the array not in the 0th column has exactly two choices of values given the value of the entry to the southwest, we may consider each array entry $\alpha_{i,j}$ with $j\ge 1$ to symbolize the outcome of the game between $i$ and $i+j$ in a tournament. If $\alpha_{i,j} = \alpha_{i+1,j-1}$ say the outcome of the game between $i$ and $i+j$ is an upset, and if not, the outcome is not an upset. Thus we have a bijection between tournaments on $n$ vertices and the arrays $Y_n^+(\{b,r,(g)\})$. Then Proposition~\ref{prop:yplusbij} completes the bijection between tournaments on $n$ vertices and the order ideals of $T_n(\{b,r,(g)\})$. 

\begin{figure}[htbp]
\[
\begin{array}{ccc}
1&1 &1 \\
2&2& \\
3&&
\end{array}
\hspace{.3cm}
\begin{array}{ccc}
1&1 &\textcolor{red}{2} \\
2&2&\\
3&&
\end{array}
\hspace{.3cm}
\begin{array}{ccc}
1&1 &2 \\
2&\textcolor{red}{3}&\\
3&&
\end{array}
\hspace{.3cm}
\begin{array}{ccc}
1&1 & \textcolor{red}{3} \\
2&\textcolor{red}{3}&\\
3&&
\end{array}
\hspace{.3cm}
\begin{array}{ccc}
1&\textcolor{red}{2} & 1 \\
2&2&\\
3&&
\end{array}
\hspace{.3cm}
\begin{array}{ccc}
1&\textcolor{red}{2} & \textcolor{red}{2} \\
2&2&\\
3&&
\end{array}
\hspace{.3cm}
\begin{array}{ccc}
1&\textcolor{red}{2} & 2 \\
2&\textcolor{red}{3}&\\
3&&
\end{array}
\hspace{.3cm}
\begin{array}{ccc}
1&\textcolor{red}{2} & \textcolor{red}{3} \\
2&\textcolor{red}{3}&\\
3&&
\end{array}
\]

\begin{center}
\includegraphics[scale=0.28]{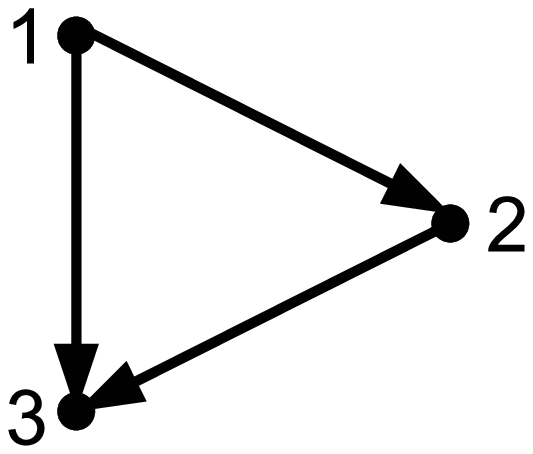}
\includegraphics[scale=0.28]{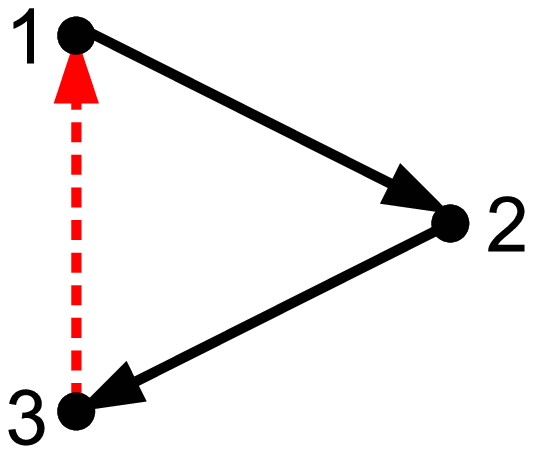}
\includegraphics[scale=0.28]{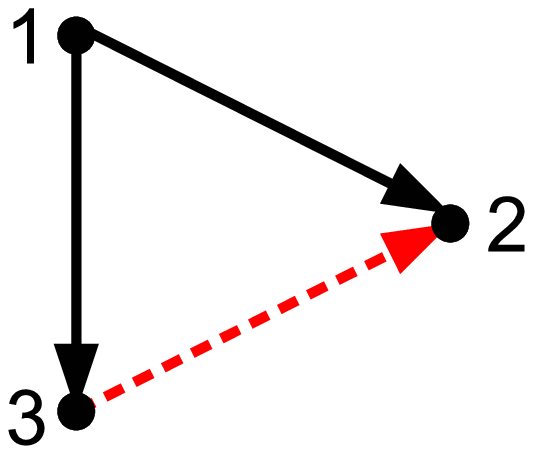}
\includegraphics[scale=0.28]{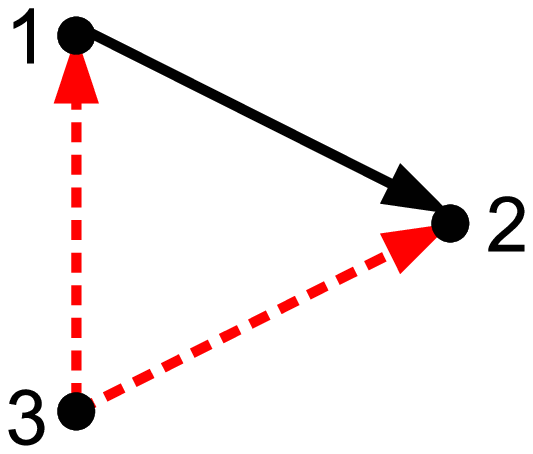}
\includegraphics[scale=0.28]{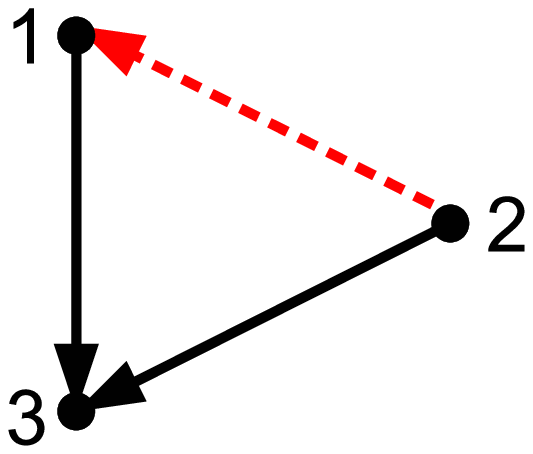}
\includegraphics[scale=0.28]{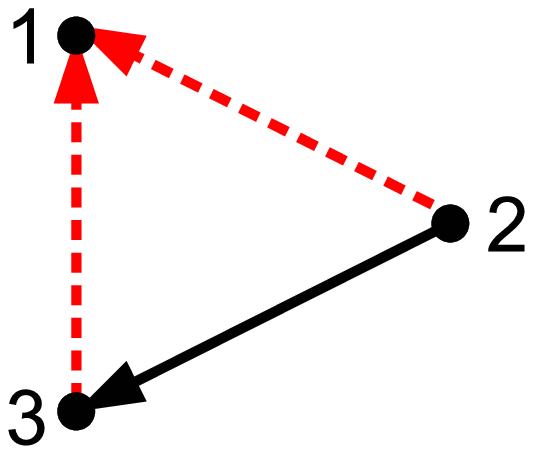}
\includegraphics[scale=0.28]{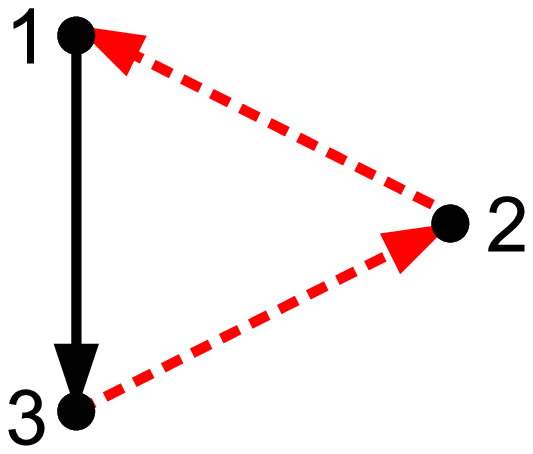}
\includegraphics[scale=0.28]{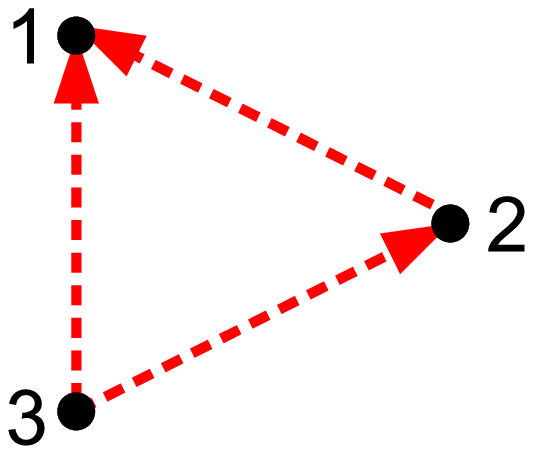}
\end{center}
\caption[The bijection between $Y_3^+(\{b,r,(g)\})$ and tournaments]{The bijection between the arrays $Y_3^+(\{b,r,(g)\})$ and tournaments on three vertices. The tournament upsets (dashed edges) correspond to the array entries equal to their southwest diagonal neighbor.}
\label{fig:brgtourn}
\end{figure}

The bijection between the order ideals of the connected three color poset $T_n(\{o,y,g\})$ and tournaments on $n$ vertices is due to Sundquist in~\cite{SUNDQUIST} and involves repeated use of jeu de taquin and column deletion to go from $Y_n(\{o,y,g\})$ (which are SSYT of shape $\delta_n$ and entries at most $n$) to objects which he calls tournament tableaux. 

We now describe Sundquist's bijection.
Begin with a SSYT $\alpha$ of shape $\delta_n$  and entries at most $n$ (that is, an element of $Y_n(\{o,y,g\})$. If there are $k$ 1s in $\alpha$ then apply jeu de taquin $k$ times to $\alpha$. If the resulting shape is not $\delta_{n-1}$ apply column deletion to the elements of $\delta_n \backslash \delta_{n-1}$ beginning at the bottom of the tableau. Put the elements removed from $\alpha$ in either step in increasing order in the first row of a new tableau $\beta$. 

\begin{figure}[htbp]
\centering
\includegraphics[scale=0.6]{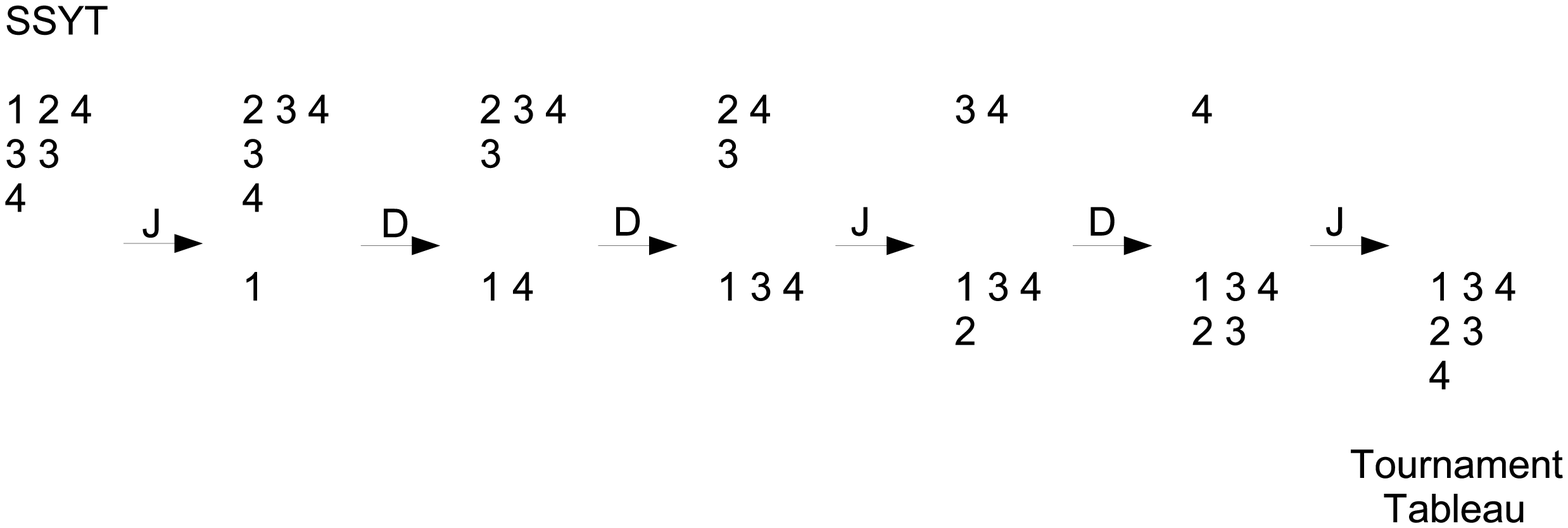}
\caption[Sundquist's bijection between SSYT and tournament tableaux.]{An example of Sundquist's bijection between a SSYT of shape $\delta_4$ and a tournament tableau. The resulting tournament tableau represents the tournament on 4 vertices with upsets between vertices 1 \& 3, 1 \& 4, 2 \& 3, and 3 \& 4.}
\label{fig:sundbij}
\end{figure}

Continue this process in like manner. At the $i$th iteration remove the entries equal to $i$ from $\alpha$ by repeated use of jeu de taquin, then bump by column deletion the elements of $\delta_{n-i+1} \backslash \delta_{n-i}$. Then make all the elements removed or bumped out of $\alpha$ during the $i$th iteration into the $i$th row of $\beta$. $\beta$ is called a tournament tableau and has the property that in the $i$th row the elements are greater than or equal to $i$ and all elements strictly greater than $i$ in row $i$ must not be repeated in that row. Thus there is an upset between $i$ and $j$ in the corresponding tournament ($i<j$) if and only if $j$ appears in row $i$ of the tournament tableau. See Figure~\ref{fig:sundbij} for an example of this bijection, where the arrows labeled J are where jeu de taquin was applied and the arrows labeled D are where column deletion was applied.

\section{The four color posets---ASMs and TSSCPPs}
\label{sec:four}

In this section we prove Theorem~\ref{thm:4four} about the number of the order ideals of the four color posets by way of two propositions. The first proposition shows that $n\times n$ ASMs are in bijection with the order ideals $J(T_n(\{b,y,o,g\}))$, and the second proposition shows that TSSCPPs inside a $2n\times 2n\times 2n$ box are in bijection with $J(T_n(\{r,g,o,(y)\}))$. Now, there is only one ASM poset inside $T_n$ but there are actually six different sets of colors $S$ such that $T_n(S)$ (or $T_n^*(S)$) is isomorphic to the TSSCPP poset, $T_n(\{r,g,o,(y)\})$.  Thus once we have the correspondence between $J(T_n(\{r,g,o,(y)\}))$ and TSSCPPs inside a $2n\times 2n\times 2n$ box, we obtain the bijection for the other five sets of colors through poset isomorphism. In this section we also discuss a generalization of the $q$-ification of the ASM counting formula~(\ref{eq:product}) which contains both the ASM and Catalan numbers as specializations.

Given an $n \times n$ ASM $A$ it is customary to consider the following bijection to objects called monotone triangles of order $n$~\cite{BRESSOUDBOOK}.
For each row of $A$ note which columns have a partial sum (from the top) of 1 in that row. Record the numbers of the columns in which this occurs 
in increasing order. This gives a triangular array of numbers 1 to $n$. 
This process can be easily reversed, and is thus a bijection. Monotone triangles can be defined as objects in their own right as follows~\cite{BRESSOUDBOOK}. 

\begin{definition}
Monotone triangles of order $n$
are all triangular arrays of integers with bottom row 1~2~3~\ldots~$n$ 
and integer entries $a_{ij}$ such that $a_{i,j} \le a_{i-1,j} \le a_{i,j+1} \mbox{ and } a_{ij} < a_{i,j+1}$.
\end{definition}

We construct a bijection between monotone triangles of order $n$ and the arrays $Y_n^+(\{b,y,o,g\})$ in the following manner (see Figure~\ref{fig:mtasmbija}).
If we rotate the monotone triangle clockwise by $\frac{\pi}{4}$ we obtain a semistandard Young tableau of staircase shape $\delta_{n+1}$ with entries at most $n$ whose northeast to southwest diagonals are weakly increasing. These are the arrays $Y_n^+(\{b,y,o,g\})$ (see Definition~\ref{def:yplus}).  
Thus we have the following proposition.

\begin{figure}[htbp]
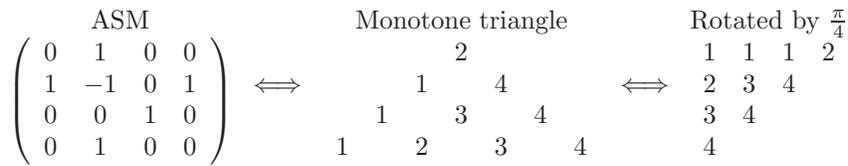

\[
\begin{array}{c}
\mbox{ASM}\\
\left( \begin{array}{cccc}
0 & 1 & 0 & 0\\
1 & -1 & 0 & 1\\
0 & 0 & 1 & 0\\
0 & 1 & 0 & 0\end{array} \right)\end{array}
\Longleftrightarrow
\begin{array}{c}
\mbox{Monotone triangle}\\
\begin{array}{ccccccc}
  & & & 2 & & & \\
  & & 1 & & 4 & & \\
  & 1 & & 3 & & 4 & \\
  1 & & 2 & & 3 & & 4\end{array} \end{array}
\Longleftrightarrow
\begin{array}{c}
\mbox{Rotated by $\frac{\pi}{4}$}\\
\begin{array}{cccc}
1&1&1&2\\
2&3&4&\\
3&4&&\\
4&&
\end{array}\end{array}
\]
\caption{An example of the bijection between $n\times n$ ASMs and $Y_n^+(\{b,y,o,g\})$}
\label{fig:mtasmbija}
\end{figure}

\begin{proposition}
\label{prop:asmbij}
The order ideals $J(T_n(\{b,y,o,g\}))$ are in bijection with $n\times n$ alternating sign matrices. 
\end{proposition}
\begin{proof}
From the above discussion we see that $n\times n$ ASMs are in bijection with the arrays $Y_n^+(\{b,y,o,g\})$. Then by Proposition~\ref{prop:yplusbij}, the arrays $Y_n^+(\{b,y,o,g\})$ are in bijection with the order ideals $J(T_n(\{b,y,o,g\}))$.
\end{proof}

\begin{figure}[htbp]
\[
\begin{array}{cc}
1 &1 \\
2&
\end{array}
\hspace{.7cm}
\begin{array}{cc}
1 &1 \\
3&
\end{array}
\hspace{.7cm}
\begin{array}{cc}
1 &2 \\
2&
\end{array}
\hspace{.7cm}
\begin{array}{cc}
1 & 2 \\
3&
\end{array}
\hspace{.7cm}
\begin{array}{cc}
2 & 2 \\
3&
\end{array}
\hspace{.7cm}
\begin{array}{cc}
1 & 3 \\
3&
\end{array}
\hspace{.7cm}
\begin{array}{cc}
2 & 3 \\
3&
\end{array}
\]
\caption{The arrays $Y_3(\{b,y,o,g\})$ which are in bijection with $3\times 3$ ASMs}
\label{fig:byog3}
\end{figure}

The lattice of monotone triangles of order $n$ (which, as we have seen, is equivalent to $J(T_n(\{b,y,o,g\}))$) is the smallest lattice containing the strong Bruhat order on the symmetric group as a subposet, that is, it is the MacNeille completion of the Bruhat order~\cite{TREILLIS}. This relationship was used in~\cite{READING_ORDER_DIMENSION} to calculate the order dimension of the strong Bruhat order in type $A_n$. The order dimension of a poset $P$ is the smallest $d$ such that $P$ can be embedded as a subposet of $\mathbb{N}^d$ with componentwise order~\cite{TROTTER}. The proof of the order dimension of $A_n$ proceeds by looking at, in our notation, the poset $T_n(\{g,b\})$, noting that each component is the product of chains, and using these chains to form a symmetric chain decomposition. Then (by another theorem in~\cite{READING_ORDER_DIMENSION}) the order dimension equals the number of chains in this symmetric chain decomposition, thus the order dimension of $A_n$ under the strong Bruhat order equals $\lfloor\frac{(n+1)^2}{4}\rfloor$.

The rank generating function $F(J(T_n(\{b,y,o,g\})),q)$ is not equal to the direct $q$-ification of the product formula (\ref{eq:product}) and does not factor nicely. The previously discussed class of tetrahedral posets whose number of order ideals factors but whose rank generating function does not factor is the two color class in which $T_n(\{b,g\})$ belongs~(see the proof of Theorem~\ref{thm:twoadj} in Section~\ref{sec:two}). The number of order ideals of posets in this class is given by a product of Catalan numbers. The rank generating function for the lattice of order ideals of these posets is not given by a product formula but rather by a product of the Carlitz-Riordan $q$-Catalan numbers which are defined by a recurrence rather than an explicit formula. This gives hope that even though $F(J(T_n(\{b,y,o,g\})),q)$ does not factor nicely, perhaps it may be found to satisfy some sort of recurrence. See Figure~\ref{fig:byog3} for the arrays $Y_3(\{b,y,o,g\})$ corresponding to $3\times 3$ ASMs.

The other six posets $J(T_n(S))$ with $S$ admissible and $|S|=4$ are each in bijection with totally symmetric self-complementary plane partitions (see Definition~\ref{def:tsscpp}). We give this bijection for $S=\{g,r,o,(y)\}$ and infer the other bijections through poset isomorphism. See Figure~\ref{fig:ryog3} for the six different sets of arrays each corresponding to TSSCPPs inside a $6\times 6\times 6$ box.

Because of the symmetry conditions, there are relatively few boxes in any TSSCPP which determine it. So we can restrict to a fundamental domain in the following manner (see Figure~\ref{fig:tsscppbijb}). 
Given a TSSCPP $t=\{t_{i,j}\}_{1\le i,j\le 2n}$ take a fundamental domain consisting of the triangular array of integers $\{t_{i,j}\}_{n+1\le i\le j\le 2n}$. In this triangular array $t_{i,j}\ge t_{i+1,j}\ge t_{i+1,j+1}$ since $t$ is a plane partition. Also for these values of $i$ and $j$ the entries $t_{i,j}$ satisfy
$0\le t_{i,j}\le 2n+1-i$.  Now if we reflect this array about a vertical line then rotate clockwise by $\frac{\pi}{4}$ we obtain a staircase shape array $x$ whose entries $x_{i,j}$ satisfy the conditions $x_{i,j}\le x_{i,j+1}\le x_{i+1,j}$ and $0\le x_{i,j}\le j$. The set of all such arrays is equivalent to $X_n(\{g,r,o,(y)\})$ (see Definition~\ref{def:aofs}). Now add $i$ to each entry in row~$i$ to obtain the arrays $Y_n^+(\{g,r,o,(y)\})$. This gives us the following proposition.

\begin{figure}[htbp]
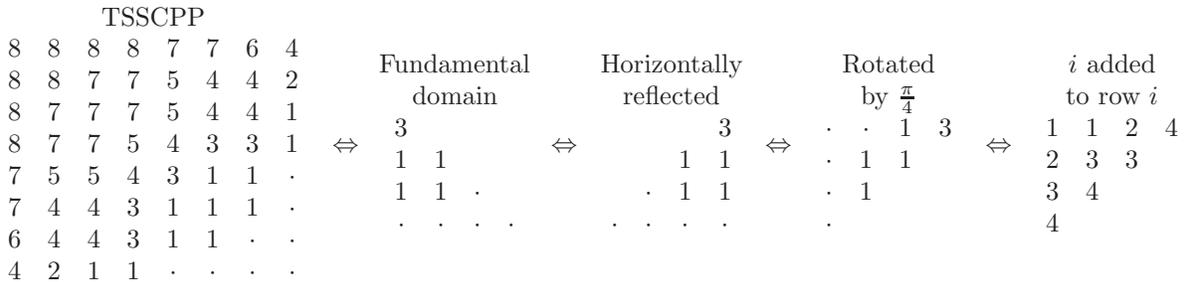

\[
\begin{array}{c}
\mbox{TSSCPP}\\
\begin{array}{cccccccc}
8&8&8&8&7&7&6&4\\
8&8&7&7&5&4&4&2\\
8&7&7&7&5&4&4&1\\
8&7&7&5&4&3&3&1\\
7&5&5&4&3&1&1&\cdot\\
7&4&4&3&1&1&1&\cdot\\
6&4&4&3&1&1&\cdot&\cdot\\
4&2&1&1&\cdot&\cdot&\cdot&\cdot
\end{array}
\end{array}
\Leftrightarrow
\begin{array}{c}
\mbox{Fundamental}\\
\mbox{domain}\\
\begin{array}{cccc}
3&&&\\
1&1&&\\
1&1&\cdot&\\
\cdot&\cdot&\cdot&\cdot
\end{array}
\end{array}
\Leftrightarrow
\begin{array}{c}
\mbox{Horizontally}\\
\mbox{reflected}\\
\begin{array}{cccc}
&&&3\\
&&1&1\\
&\cdot&1&1\\
\cdot&\cdot&\cdot&\cdot
\end{array}
\end{array}
\Leftrightarrow
\begin{array}{c}
\mbox{Rotated}\\
\mbox{by $\frac{\pi}{4}$}\\
\begin{array}{cccc}
\cdot&\cdot&1&3\\
\cdot&1&1&\\
\cdot&1&&\\
\cdot&&&
\end{array}
\end{array}
\Leftrightarrow
\begin{array}{c}
\mbox{$i$ added}\\
\mbox{to row $i$}\\
\begin{array}{cccc}
1&1&2&4\\
2&3&3&\\
3&4&&\\
4&&&
\end{array}\end{array}
\]
\caption[An example of the bijection between TSSCPPs and $Y_n^+(\{r,g,o,(y)\})$]{An example of the bijection between TSSCPPs inside a $2n\times 2n\times 2n$ box and $Y_n^+(\{r,g,o,(y)\})$}
\label{fig:tsscppbijb}
\end{figure}

\begin{proposition}
\label{prop:tsscppbij}
The order ideals $J(T_n(\{r,g,o,(y)\}))$ are in bijection with totally symmetric self complementary plane partitions inside a $2n\times 2n\times 2n$ box. 
\end{proposition}
\begin{proof}
From the above discussion we see that TSSCPPs are in bijection with the arrays $Y_n^+(\{r,g,o,(y)\})$. Then by Proposition~\ref{prop:yplusbij}, $Y_n^+(\{r,g,o,(y)\})$ is in bijection with $J(T_n(\{r,g,o,(y)\}))$.
\end{proof}

\begin{figure}[htbp]
\begin{eqnarray*}
Y_3(\{r,g,o,(y)\}) &=& 
\left\{
\begin{array}{cc}
1 &1 \\
2&
\end{array},
\begin{array}{cc}
1 &2 \\
2&
\end{array},
\begin{array}{cc}
1 &2 \\
3&
\end{array},
\begin{array}{cc}
1 & 3 \\
2&
\end{array},
\begin{array}{cc}
2 & 2 \\
3&
\end{array},
\begin{array}{cc}
1 & 3 \\
3&
\end{array},
\begin{array}{cc}
2 & 3 \\
3&
\end{array}
\right\}
\\
&&\\
Y_3(\{r,b,(g),y\}) &=&
\left\{
\begin{array}{cc}
1 &1 \\
2&
\end{array},
\begin{array}{cc}
1 &2 \\
2&
\end{array},
\begin{array}{cc}
1 &2 \\
3&
\end{array},
\begin{array}{cc}
2 & 2 \\
2&
\end{array},
\begin{array}{cc}
2 & 2 \\
3&
\end{array},
\begin{array}{cc}
1 & 3 \\
3&
\end{array},
\begin{array}{cc}
2 & 3 \\
3&
\end{array}
\right\}
\\
&&\\
Y_3(\{y,s,(g),r\}) &=&
\left\{
\begin{array}{cc}
1 &1 \\
2&
\end{array},
\begin{array}{cc}
1 &2 \\
2&
\end{array},
\begin{array}{cc}
1 &2 \\
3&
\end{array},
\begin{array}{cc}
2 & 2 \\
2&
\end{array},
\begin{array}{cc}
2 & 2 \\
3&
\end{array},
\begin{array}{cc}
2 & 3 \\
2&
\end{array},
\begin{array}{cc}
2 & 3 \\
3&
\end{array}
\right\}
\\
&&\\
Y_3(\{y,s,(g),b\}) &=&
\left\{
\begin{array}{cc}
1 &1 \\
2&
\end{array},
\begin{array}{cc}
1 &2 \\
2&
\end{array},
\begin{array}{cc}
1 &1 \\
3&
\end{array},
\begin{array}{cc}
1 & 2 \\
3&
\end{array},
\begin{array}{cc}
2 & 2 \\
2&
\end{array},
\begin{array}{cc}
2 & 2 \\
3&
\end{array},
\begin{array}{cc}
2 & 3 \\
3&
\end{array}
\right\}
\\
&&\\
Y_3(\{o,s,(b),g\}) &=&
\left\{
\begin{array}{cc}
1 &1 \\
2&
\end{array},
\begin{array}{cc}
1 &2 \\
2&
\end{array},
\begin{array}{cc}
1 &1 \\
3&
\end{array},
\begin{array}{cc}
1 & 2 \\
3&
\end{array},
\begin{array}{cc}
2 & 1 \\
3&
\end{array},
\begin{array}{cc}
2 & 2 \\
3&
\end{array},
\begin{array}{cc}
2 & 3 \\
3&
\end{array}
\right\}
\\
&&\\
Y_3(\{r,b,(g),s\}) &=&
\left\{
\begin{array}{cc}
1 &1 \\
2&
\end{array},
\begin{array}{cc}
1 &2 \\
2&
\end{array},
\begin{array}{cc}
2 &1 \\
2&
\end{array},
\begin{array}{cc}
1 & 2 \\
3&
\end{array},
\begin{array}{cc}
2 & 2 \\
2&
\end{array},
\begin{array}{cc}
2 & 2 \\
3&
\end{array},
\begin{array}{cc}
2 & 3 \\
3&
\end{array}
\right\}
\end{eqnarray*}
\caption[The six sets of arrays in bijection with TSSCPPs for $n=3$]{
The six sets of arrays which are in bijection with TSSCPPs inside a $6\times 6\times 6$ box}
\label{fig:ryog3}
\end{figure}

%
Equipped with Propositions~\ref{prop:asmbij} and~\ref{prop:tsscppbij} we are now ready to prove Theorem~\ref{thm:4four}.

\begin{thm219}
If $S$ is an admissible subset of $\{r,b,g,o,y,s\}$ and $|S|=4$ then 
\begin{equation}
\label{eq:4}
|J(T_n(S))|=\displaystyle\prod_{j=0}^{n-1} \frac{(3j+1)!}{(n+j)!}=\displaystyle\prod_{1\le i\le j\le k\le n-1} \frac{i+j+k+1}{i+j+k-1}.
\end{equation}
\end{thm219}
\begin{proof}
The posets $T_n(S)$ for $S$ admissible and $|S|=4$ are $T_n(\{g,y,b,o\})$, the three isomorphic posets 
$T_n(\{r,o,(y),g\})$, $T_n(\{r,b,(g),y\})$, and $T_n(\{y,s,(g),r\})$, and the three posets dual to these, $T_n(\{y,s,(g),b\})$, $T_n(\{o,s,(b),g\})$, $T_n(\{r,b,(g),s\})$. In Proposition~\ref{prop:asmbij} we showed that the order ideals of $T_n(\{g,y,b,o\})$ are in bijection with $n\times n$ ASMs and in Proposition~\ref{prop:tsscppbij} we showed that the order ideals of $T_n(\{r,o,(y),g\})$ are in bijection with TSSCPPs inside a $2n\times 2n\times 2n$ box.  Therefore by poset isomorphism TSSCPPs inside a $2n\times 2n\times 2n$ box are in bijection with the order ideals of any of  $T_n(\{r,o,(y),g\})$, $T_n(\{r,b,(g),y\})$, $T_n(\{y,s,(g),r\})$, $T_n^*(\{y,s,(g),b\})$, $T_n^*(\{o,s,(b),g\})$, or  $T_n^*(\{r,b,(g),s\})$. Thus by the enumeration of ASMs in~\cite{ZEILASM} and \cite{KUP_ASM_CONJ} and the enumeration of TSSCPPs in~\cite{ANDREWS_PPV} we have the above formula for the number of order ideals.
\end{proof}

As in the two-color Catalan case of Theorem~\ref{thm:twoadj}, the direct $q$-ification of the counting formula~(\ref{eq:4}), does not equal the rank generating function of either ASMs or TSSCPPs, but there may be a way to find a weight on ASMs or TSSCPPs which would yield the $q$-product formula as generating function. An interesting generalization of the $q$-product formula which has connections with the MacMahon $q$-Catalan numbers is found in the thesis of Thomas Sundquist~\cite{SUNDQUIST}. Sundquist studies the following product formula:
\begin{equation}
\label{eq:qpasm}
A(n,p;q)= \frac{s_{(p\delta_n)'}(n p)}{s_{p\delta_n}(n)} q^{-{p\choose 2}\sum_{i=1}^{n-1} i^2}=\displaystyle\prod_{k=0}^{n-1} \frac{(np+k)!_q k!_q}{(kp+k+p)!_q (pk+k)!_q}
\end{equation}
where $s_\lambda$ is the Schur function and $\delta_n$ is the staircase shape partition.

$A(n,2;q)$ reduces to the $q$-ified ASM product formula.
\[
A(n,2;q)=\displaystyle\prod_{k=0}^{n-1} \frac{(3k+1)!_q}{(n+k)!_q}
\]

$A(2,p;q)$ reduces to the MacMahon $q$-Catalan numbers.
\[
A(2,p;q)=\frac{1}{[p+1]_q}{2p\brack p}_q
\]

Sundquist was able to prove that $A(n,p;q)$ is a polynomial in $q$, but he was unable to find a combinatorial interpretation (even for $A(n,p;1)$) for general $n$ and $p$. It is well-known that $\frac{1}{[p+1]_q}{2p\brack p}_q$ for fixed $p$ is the generating function for Catalan sequences (or Dyck paths) weighted by major index. $A(n,p;q)$ suggests a kind of generalization of major index to a larger class of objects including those counted by the $q$-ASM numbers, but that larger class of objects is still unknown as is an interpretation for this $q$-weight on ASMs or TSSCPPs. 

$A(n,p;1)$ is known to have connections to physics, however.
An expression equal to $A(n,p;1)$ appears in~\cite{QKZJOSEPH} (Equation (5.5)) and comes from a weighted sum of components of the ground state of the $A_{k-1}$ IRF model in statistical physics. In this paper the authors comment that a combinatorial interpretation for $A(n,p;1)$ as a generalization of ASMs would be very interesting.

\section{Connections between ASMs, TSSCPPs, and tournaments}
\label{sec:3to4}

In this section we discuss the expansion of the tournament generating function as a sum over ASMs and derive a new expansion of the tournament generating function as a sum over TSSCPPs. We discuss the implications of the combination of these two expansions and also describe which subsets of tournaments correspond to TSSCPPs and ASMs. 

The alternating sign matrix conjecture first arose when David Robbins and Howard Rumsey expanded a generalization of the determinant of a square matrix called the $\lambda$-determinant as a sum over all $n\times n$ alternating sign matrices (\cite{ROBBINSRUMSEY}) and wondered how many elements would be in the sum. When one combines this $\lambda$-determinant formula with the Vandermonde determinant, 
one obtains the following theorem. (This theorem is implicit in~\cite{ROBBINSRUMSEY}, but appears explicitly in~\cite{BRESSOUDBOOK}.) In this theorem we will need the notion of inversion number for ASMs. The inversion number of an ASM $A$ is defined as $I(A)=\sum A_{ij} A_{k\ell}$ where the sum is over all $i,j,k,\ell$ such that $i>k$ and $j<\ell$. This definition extends the usual notion of inversion in a permutation matrix. 

\begin{theorem}[Robbins-Rumsey]
\label{thm:lambdadet}
Let $A_n$ be the set of $n\times n$ alternating sign matrices, and for $A\in A_n$ let $I(A)$ denote the inversion number of $A$ and $N(A)$ the number of $-1$ entries in $A$, then
\begin{equation}
\label{eq:2enumeq}
\displaystyle\prod_{1\le i < j\le n} (x_i + \lambda x_j) = \displaystyle\sum_{A\in A_n} \lambda^{I(A)}\left(1+\lambda^{-1}\right)^{N(A)}\displaystyle\prod_{i,j=1}^n x_{j}^{(n-i)A_{ij}}.
\end{equation}
\end{theorem}

Note that there are ${n\choose 2}$ factors in the product on the left-hand side, thus 
$\prod_{1\le i < j\le n} (x_i + \lambda x_j)$ is the generating function for tournaments on $n$ vertices where each factor of $(x_i + \lambda x_j)$ represents the outcome of the game between $i$ and $j$ in the tournament. If $x_i$ is chosen then the expected winner, $i$, is the actual winner, and if $\lambda x_j$ is chosen then $j$ is the unexpected winner and the game is an upset. Thus in each monomial in the expansion of $\prod_{1\le i < j\le n}(x_i + \lambda x_j)$ the power of $\lambda$ equals the number of upsets and the power of $x_k$ equals the number of wins of $k$ in the tournament corresponding to that monomial.

We rewrite Theorem~\ref{thm:lambdadet} in different notation which will also be needed later.
For any staircase shape integer array $\alpha\in Y_n^+(S)$ (recall Definition~\ref{def:yplus}) let $E_{i,k}(\alpha)$ be the number of entries of value $k$ in row $i$ equal to their southwest diagonal neighbor, $E^i(\alpha)$ be the number of entries in (southwest to northeast) diagonal $i$ equal to their southwest diagonal neighbor, and $E_i(\alpha)$ be the number of entries in row $i$ equal to their southwest diagonal neighbor, that is, $E_i(\alpha)=\sum_k E_{i,k}(\alpha)$. Also let $E(\alpha)$ be the total number of entries of $\alpha$ equal to their southwest diagonal neighbor, that is, $E(\alpha)=\sum_{i} E_{i}(\alpha)=\sum_{i} E^{i}(\alpha)$. We now define variables for the content of $\alpha$. Let $C_{i,k}(\alpha)$ be the number of entries in row $i$ with value $k$ and let $C_k(\alpha)$ be the total number of entries of $\alpha$ equal to $k$, that is, $C_{k}(\alpha)=\sum_i C_{i,k}(\alpha)$. Let $N(\alpha)$ be the number of entries of $\alpha$ strictly greater than their neighbor to the west and strictly less than their neighbor to the southwest. When $\alpha\in Y_n^+(\{b,y,o,g\})$ then $N(\alpha)$ equals the number of $-1$ entries in the corresponding ASM. For example, let $\alpha$ be the following array in $Y_5^+(\{b,y,o,g\})$ which is in bijection with the ASM $A$ shown below.

\[
\begin{array}{ccccc}
1&1&2&2&3\\
2&2&3&3&\\
3&4&4&&\\
4&5&&&\\
5&&&&
\end{array}
\Longleftrightarrow
\left( \begin{array}{rrrrr}
0 & 0 & 1 & 0 & 0\\
0 & 1 & 0 & 0 & 0\\
0 & 0 & 0 & 1 & 0\\
1 & 0 & -1 & 0 & 1\\
0&0&1&0&0\end{array} \right)
\]

Then the parameters on the array $\alpha$ are as follows.
$E_{1,2}=E_{1,3}=E_{3,4}=E_{4,5}=1$ and $E_{i,j}=0$ for all other choices of $i$ and $j$.
The number of diagonal equalities in each row of $\alpha$ is given by $E_1=2, E_2=0, E_3=1, E_4=1, E_5=0$.
The number of diagonal equalities in each column of $\alpha$ is given by $E^1=0, E^2=0, E^3=1, E^4=1, E^5=2, E=4$. The number of $-1$ entries in $A$ is $N=1$. 
The content parameters on $\alpha$ are given by $C_{1,1}=2, C_{1,2}=2, C_{1,3}=1, C_{2,2}=2, C_{2,3}=2,
C_{3,3}=1$,  $C_{3,4}=2, C_{4,4}=1, C_{4,5}=1, C_{5,5}=1,
C_1=2, C_2=4, C_3=4, C_4=3, C_5=2$.

Using this notation we reformulate Theorem~\ref{thm:lambdadet} in the following way.

\begin{theorem}
\label{thm:lambdaASM}
The generating function for tournaments on $n$ vertices can be expanded as a sum over the ASM arrays $Y_n^+(\{b,y,o,g\})$ in the following way.
\begin{equation}
\label{eq:asmlambda}
\displaystyle\prod_{1\le i < j\le n} (x_i + \lambda x_j) = \displaystyle\sum_{\alpha\in Y_n^+(\{b,y,o,g\})} \lambda^{E(\alpha)} (1 + \lambda)^{N(\alpha)} \displaystyle\prod_{k=1}^n x_k^{C_k(\alpha)-1}.
\end{equation}
\end{theorem}
\begin{proof}
First we rewrite Equation~\ref{eq:2enumeq} by factoring out $\lambda^{-1}$ from each $\left(1+\lambda^{-1}\right)$.
\[
\displaystyle\prod_{1\le i < j\le n} (x_i + \lambda x_j) = \displaystyle\sum_{A\in A_n} \lambda^{I(A) - N(A)}\left(1+\lambda\right)^{N(A)}\displaystyle\prod_{i,j=1}^n x_{j}^{(n-i)A_{ij}}.
\]
Let $\alpha\in Y_n^+(\{b,y,o,g\})$ be the array which corresponds to $A$. It is left to show that $I(A)-N(A)=E(\alpha)$ and $\prod_{i,j=1}^n x_{j}^{(n-i)A_{ij}}=\prod_{j=1}^n x_j^{C_j(\alpha)-1}$. In the latter equality take the product over $i$ of the left hand side: $\prod_{i,j=1}^n x_{j}^{(n-i)A_{ij}}=\prod_{j=1}^n x_{j}^{\sum_{i=1}^n(n-i)A_{ij}}$. 
We wish to show $C_{j}(\alpha)-1=\sum_{i=1}^n (n-i)A_{ij}$. $C_j(\alpha)$ equals the number of entries of $\alpha$ with value $j$, so $C_j(\alpha)-1$ equals the number of entries of $\alpha$ with value $j$ not counting the $j$ in the 0th column. Using the bijection between ASMs and the arrays $Y_n(\{b,y,o,g\})$ we see that the number of $j$s in columns 1 through $n-1$ of $\alpha$ equals the number of 1s in column $j$ of $A$ plus the number of zeros in column $j$ of $A$ which are south of a 1 with no $-1$s in between. This is precisely what $\sum_{i=1}^n(n-i)A_{ij}$ counts by taking a positive contribution from every~1 and every entry below that~1 in column $j$ and then subtracting one for every~$-1$ and every entry below that $-1$ in column $j$. Thus $C_{j}(\alpha)-1=\sum_{i=1}^n (n-i)A_{ij}$ so that $\prod_{i,j=1}^n x_{j}^{(n-i)A_{ij}}=\prod_{j=1}^n x_j^{C_j(\alpha)-1}$.

Now we wish to show that $I(A)-N(A)=E(\alpha)$. 
Recall that the inversion number of an ASM is defined as $I(A)=\sum A_{ij} A_{k\ell}$ where the sum is over all $i,j,k,\ell$ such that $i>k$ and $j<\ell$. Fix $i$, $j$, and $\ell$ and consider $\sum_{k<i} A_{ij} A_{k\ell}$. Let $k'$ be the row of the southernmost nonzero entry in column $\ell$ such that $k'<i$. If there exists no such $k'$ (that is, $A_{k\ell}=0 \mbox{ }\forall\mbox{ } k<i$) or if $A_{k'\ell}=-1$ then $\sum_{k>i} A_{ij} A_{k\ell}=0$ since there must be an even number of nonzero entries in $\{A_{k\ell}, k<i\}$ half of which are 1 and half of which are $-1$. If $A_{k'\ell}=1$ then $\sum_{k<i} A_{ij} A_{k\ell}=A_{ij}$. Thus $I(A)=\sum_{i,j} \alpha_{ij} A_{ij}$ where $\alpha_{ij}$ equals the number of columns east of column $j$ such that $A_{k'\ell}$ with $k'>i$ exists and equals~1. Let column $\ell'$ be one of the columns counted by $\alpha_{ij}$. Then $A_{i\ell'}$ cannot equal~1, otherwise $A_{k'\ell'}$ would either not exist or equal $-1$. If $A_{i\ell'}=0$ then in $\alpha$ there is a corresponding diagonal equality. If $A_{i\ell'}=-1$ then there is no diagonal equality in $\alpha$. Thus $I(A)=E(\alpha)+N(A)$.
\end{proof}

%

Recall that the left-hand side of~(\ref{eq:asmlambda}) is the generating function for tournaments where the power of $\lambda$ equals the number of upsets and the power of $x_k$ equals the number of wins of $k$. If we set $\lambda =1$ and take the principle specialization of the $x_k$'s this yields the generating function of Theorem~\ref{thm:3nadj} multiplied by a power of $q$. Furthermore, if we set $\lambda = 1$ and $x_i=1$ for all $i$ we have the following corollary.
\begin{corollary}[Mills-Robbins-Rumsey]
\label{cor:2nc2asm}
\begin{equation}
\label{eq:asm2enum}
2^{n\choose 2} = \displaystyle\sum_{A\in A_n} 2^{N(A)}.
\end{equation}
\end{corollary}
That is, $2^{n\choose 2}$ is the 2-enumeration of ASMs with respect to the number of $-1$s. This result appears in~\cite{MRRASMDPP} and~\cite{ROBBINSRUMSEY}. The connection between this result and Aztec diamonds is shown in~\cite{AZTEC1}. 

Many people have wondered what the TSSCPP analogue of the $-1$ in an ASM may be. The following theorem does not give a direct analogue, but rather expands the left-hand side of~(\ref{eq:asmlambda}) as a sum over TSSCPPs instead of ASMs.

\begin{theorem}
\label{thm:lambdaTSSCPP}
The generating function for tournaments on $n$ vertices can be expanded as a sum over the TSSCPP arrays $Y_n^+(\{b,r,(g),y\})$ in the following way.
\begin{equation}
\label{eq:lambdatsscpp}
\displaystyle\prod_{1\le i < j\le n} (x_i + \lambda x_j) = \displaystyle\sum_{\alpha\in Y_n^+(\{b,r,(g),y\})} \lambda^{E(\alpha)} \displaystyle\prod_{i=1}^{n-1} x_i^{n-i-E_i(\alpha)} \displaystyle\sum_{\mbox{row shuffles $\alpha'$ of $\alpha$}} \mbox{ }\displaystyle\prod_{j=1}^{n-1} x_j^{E^j(\alpha')}
\end{equation}
\begin{equation}
\label{eq:lambdanoxtsscpp}
(1+\lambda)^{n\choose 2} = \displaystyle\sum_{\alpha\in Y_n^+(\{b,r,(g),y\})} \lambda^{E(\alpha)} \displaystyle\prod_{1\le i\le k\le n-1} {C_{i+1,k}(\alpha)\choose E_{i,k}(\alpha)}
\end{equation}
where a row shuffle $\alpha'$ of $\alpha\in Y_n^+(\{b,r,(g),y\}$ is an array obtained by reordering the entries in the rows of $\alpha$ in such a way that $\alpha'\in Y_n^+(\{b,r,(g)\}$.
\end{theorem}
\begin{proof}
The idea of the proof is to begin with the set $Y_n^+(\{b,r,(g),y\})$ and remove the inequality restriction corresponding to the color yellow to obtain the arrays $Y_n^+(\{b,r,(g)\})$. We wish to prove that all the arrays in $Y_n^+(\{b,r,(g)\})$ can be obtained uniquely by the rearrangement of the rows of the arrays in $Y_n^+(\{b,r,(g),y\})$. 

First we need to see how $\prod_{1\le i < j\le n} (x_i + \lambda x_j)$ is the generating function of the arrays $Y_n^+(\{b,r,(g)\})$. Recall the bijection of Section~\ref{sec:three} between $Y_n^+(\{b,r,(g)\})$ and tournaments on $n$ vertices. Let each entry $\alpha_{i,j}$ with $j\ge 1$ of $\alpha\in Y_n^+(\{b,r,(g)\})$ contribute a weight $x_i$ if $\alpha_{i,j} = \alpha_{i+1,j-1}-1$ and $\lambda x_{j+i}$ if $\alpha_{i,j}=\alpha_{i+1,j-1}$. Thus $\prod_{1\le i < j\le n} (x_i + \lambda x_j)$ is the generating function of $Y_n^+(\{b,r,(g)\})$. 

Now we give an algorithm for turning any element $\alpha\in Y_n^+(\{b,r,(g)\})$ into an element of $Y_n^+(\{b,r,(g),y\})$ thus grouping all the elements of $Y_n^+(\{b,r,(g)\})$ into fibers over the elements of $Y_n^+(\{b,r,(g),y\})$. Assume all the rows of $\alpha$ below row $i$ are in increasing order. Thus $\alpha_{i+1,j}\le \alpha_{i+1,j+1}$. If $\alpha_{i+1,j}< \alpha_{i+1,j+1}$ then $\alpha_{i,j+1}\le \alpha_{i+1,j+2}$ since $\alpha_{i,j+1}\in \{\alpha_{i+1,j}, \alpha_{i+1,j}-1\}$ and $\alpha_{i,j+2}\in \{\alpha_{i+1,j+1},\alpha_{i+1,j+1}-1\}$ by the inequalities corresponding to red and blue. So the only entries which may be out of order in row~$i$ are those for which their southwest neighbors are equal. If $\alpha_{i+1,j} = \alpha_{i+1,j+1}$ but $\alpha_{i,j+1}>\alpha_{i,j+2}$ it must be that $\alpha_{i,j+1}=\alpha_{i+1,j}$ and $\alpha_{i,j+2}=\alpha_{i+1,j+1}-1$. So we may swap $\alpha_{i,j+1}$ and $\alpha_{i,j+2}$ along with their entire respective northeast diagonals while not violating the red and blue inequalities. 
By completing this process for all rows we obtain an array with weakly increasing rows, thus this array is actually in $Y_n^+(\{b,r,(g),y\})$. 

Now we do a weighted count of how many arrays in $Y_n^+(\{b,r,(g)\})$ are mapped to a given array in $Y_n^+(\{b,r,(g),y\})$. Again we rely on the fact that entries in a row can be reordered only when their southwest neighbors are equal. Thus to find the weight of all the $Y_n^+(\{b,r,(g)\})$ arrays corresponding to a single $Y_n^+(\{b,r,(g),y\})$ array we simply need to find the set of diagonals containing equalities. The diagonal equalities give a weight dependent on which diagonal they are in, whereas the diagonal inequalities give a weight according to their row (which remains constant). Thus if we are keeping track of the $x_i$ weight we can do no better than to write this as a sum over all the allowable (in the sense of not violating the $b$ or $r$ inequalities) shuffles of the rows of $\alpha$ with the $x$ weight of the diagonal equalities dependent on the position. Thus we have Equation (\ref{eq:lambdatsscpp}). See Figure~\ref{fig:brgylam} for an example when $n=3$.

If we set $x_i=1$ for all $i$ and only keep track of the $\lambda$ we can make a more precise statement. The above proof shows that the $\lambda$'s result from the diagonal equalities, and the number of different reorderings of the rows tell us the number of different elements of $Y_n^+(\{b,r,(g)\})$ which correspond to a given element of $Y_n^+(\{b,r,(g),y\})$. We count this number of allowable reorderings as a product over all rows $i$ and all array values $k$ as ${C_{i+1,k}(\alpha)\choose E_{i,k}(\alpha)}$. This yields Equation~(\ref{eq:lambdanoxtsscpp}).
\end{proof}

\begin{figure}[htbp]
\[
\begin{array}{ccc}
1&1 &1 \\
2&2& \\
3&&
\end{array}
\hspace{1.5cm}
\begin{array}{ccc}
1&1 &\textcolor{red}{2} \\
2&2&\\
3&&
\end{array}
\hspace{1.5cm}
\begin{array}{ccc}
1&1 &2 \\
2&\textcolor{red}{3}&\\
3&&
\end{array}
\hspace{1.5cm}
\begin{array}{ccc}
1&1 & \textcolor{red}{3} \\
2&\textcolor{red}{3}&\\
3&&
\end{array}
\]

\[
x_1^2 x_2
\hspace{1.5cm}
\textcolor{red}{\lambda} x_1 x_2 \textcolor{red}{(x_2+x_3)}
\hspace{1.5cm}
\textcolor{red}{\lambda} x_1^2 \textcolor{red}{x_3}
\hspace{1.5cm}
\textcolor{red}{\lambda^2} x_1 \textcolor{red}{x_3^2}
\]

%
\[
\begin{array}{ccc}
1&\textcolor{red}{2} & \textcolor{red}{2} \\
2&2&\\
3&&
\end{array}
\hspace{1.5cm}
\begin{array}{ccc}
1&\textcolor{red}{2} & 2 \\
2&\textcolor{red}{3}&\\
3&&
\end{array}
\hspace{1.5cm}
\begin{array}{ccc}
1&\textcolor{red}{2} & \textcolor{red}{3} \\
2&\textcolor{red}{3}&\\
3&&
\end{array}
\]

\[
\textcolor{red}{\lambda^2} x_2 \textcolor{red}{x_2 x_3}
\hspace{1.5cm}
\textcolor{red}{\lambda^2} x_1 \textcolor{red}{x_2 x_3}
\hspace{1.5cm}
\textcolor{red}{\lambda^3 x_2 x_3^2}
\]

\caption[The arrays $Y_3^+(\{b,r,(g),y\})$]{The arrays $Y_3^+(\{b,r,(g),y\})$ and their corresponding contribution toward the right hand side of Equation~(\ref{eq:lambdatsscpp}) 
}
\label{fig:brgylam}
\end{figure}

If we set $\lambda=1$ in Equation~(\ref{eq:lambdanoxtsscpp}) we obtain the following corollary.
\begin{corollary}
\label{cor:tsscpp2nc2}
\begin{equation}
\label{eq:2nc2tsscpp}
2^{n\choose 2} = \displaystyle\sum_{\alpha\in Y_n^+(\{b,r,(g),y\})} \displaystyle\prod_{1\le i\le k\le n-1} {C_{i+1,k}(\alpha)\choose E_{i,k}(\alpha)}.
\end{equation}
\end{corollary}

Corollary~\ref{cor:tsscpp2nc2} is quite different from Corollary~\ref{cor:2nc2asm}. In Corollary~\ref{cor:2nc2asm} the fibers over ASMs can be only powers of 2, whereas the fibers over TSSCPPs in Corollary~\ref{cor:tsscpp2nc2} have no such restriction. Thus there is no direct way to match up ASMs and TSSCPPs through all the tournaments in the fibers, but it may be possible to find a system of distinct representatives of tournaments which would show us how to map a fiber above a TSSCPP to a fiber above an ASM.

The difference in the weighting of ASMs and TSSCPPs in Theorems~\ref{thm:lambdaASM} and~\ref{thm:lambdaTSSCPP} is also substantial. For ASMs the more complicated part of the formula arises in the power of $\lambda$ and for TSSCPPs the complication comes from the $x$ variables. These theorems are also strangely similar. From these theorems we see that the tournament generating function can be expanded as a sum over either ASMs or TSSCPPs, but we still have no direct reason why the number of summands should be the same. 
The combination of Theorems~\ref{thm:lambdaASM} and \ref{thm:lambdaTSSCPP} may contribute toward finding a bijection between ASMs and TSSCPPs, but also shows why a bijection is not obvious.

The last thing we wish to do in this section is to describe which subsets of tournaments correspond to TSSCPPs and ASMs.
First we show that 
for a certain choice of TSSCPP poset inside $T_n$, TSSCPPs can be directly seen to be subsets of tournaments, refining the bijection of Section~\ref{sec:three} between the order ideals $J(T_n(\{r,b,(g)\}))$ and tournaments on $n$ vertices.
\begin{theorem}
\label{thm:tsscpptournbij}
TSSCPPs inside a $2n\times 2n\times 2n$ box are in bijection with tournaments on vertices labeled $1,2,\ldots,n$ which satisfy the following condition on the upsets: if vertex $v$ has $k$ upsets with vertices in $\{u,u+1,\ldots,v-1\}$ then vertex $v-1$ has at most $k$ upsets with vertices in $\{u,u+1,\ldots,v-2\}$. 
\end{theorem}
\begin{proof}
We have seen in Section~\ref{sec:three} the bijection between the order ideals of $T_n(\{r,b,(g)\})$ and tournaments on $n$ vertices. Recall that the inequalities corresponding to the colors red and blue for $\alpha\in Y_n^+(\{r,b,(g)\})$ place conditions on the diagonals such that 
there are exactly two choices for any entry $\alpha_{i,j}$ not in column 0. Either $\alpha_{i,j}=\alpha_{i+1,j-1}$ or $\alpha_{i,j}=\alpha_{i+1,j-1}-1$. In the bijection with tournaments, $\alpha_{i,j}$ tells the outcome of the game between $i$ and $i+j$ in the tournament. The outcome is an upset if $\alpha_{i,j}=\alpha_{i+1,j-1}$ and not an upset otherwise. Thus if we consider the TSSCPP arrays $Y_n^+(\{r,b,(g),y\})$ we need only find an interpretation for the color yellow in terms of tournaments. 

Recall that the color yellow corresponds to a weak increase across the rows of $\alpha$. In order for weak increase across the rows of $\alpha$ to be satisfied, 
for each choice of $i\in\{1,\ldots,n-1\}$ and $j\in\{1,\ldots,n-i-1\}$ the number of diagonal equalities to the southwest of $\alpha_{i,j}$ must be less than or equal to the number of diagonal equalities to the southwest of $\alpha_{i,j+1}$. So in terms of tournaments, the number of upsets between $i+j+1$ and vertices greater than or equal to $i$ must be greater than or equal to the number of upsets between $i+j$ and vertices greater than or equal to $i$.
\end{proof}

Note that every integer array/tournament pair in Figure~\ref{fig:brgtourn} satisfies the condition of this theorem and is thus in bijection with a TSSCPP except the pair shown below in Figure~\ref{fig:pair2}. This is because vertex $3$ has zero upsets with vertices in $\{1,2\}$ but vertex $2$ has one upset with vertices in $\{1\}$.

\begin{figure}[htbp]
\centering
\includegraphics[scale=0.5]{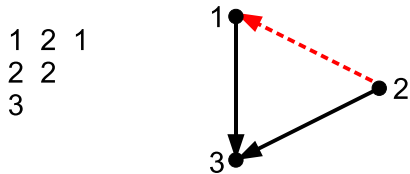}
\caption[The only tournament on 3 vertices not corresponding to a TSSCPP]{The only tournament on 3 vertices (and corresponding array in $Y_3^+(\{b,r,(g)\})$) which does not correspond to a TSSCPP}
\label{fig:pair2}
\end{figure}

Next we discuss a previously known way of viewing ASMs as a subset of tournaments. 
In the paper~\cite{HAMELKING}, Hamel and King give a weight-preserving bijection between objects called primed shifted semistandard tableaux (PST) of staircase shape which encode the 2-enumeration of ASMs and primed shifted tableaux (PD) of staircase shape which encode tournaments. Both PSTs and PDs have entries in the set $\{1,2',2,3',3,\ldots,n',n\}$. The bijection between PSTs and PDs proceeds using jeu de taquin on the primed entries. If one removes the primes from all entries of a PST, one is left with arrays equivalent to $Y_n^+\{b,y,o,g\}$ which are in bijection with $n\times n$ ASMs. Now for any entry of a PST which represents a $-1$ in the corresponding ASM, we have a choice as to whether that entry should be primed or not in the PST. Thus given an ASM, if we fix some rule concerning how to choose which entries of the ASM array should be primed in the inclusion into PSTs, we can map ASMs into tournaments by way of the bijection between PSTs and PDs. Since the bijection between PSTs and PDs uses jeu de taquin, though, it remains difficult to say which subset of tournaments is in the image of this map.
Also, since PDs encode tournaments, they are in direct bijection with the tournament tableaux of Sundquist (see Section~\ref{sec:three}). It would be interesting to look at the bijection from PSTs to PDs to tournament tableaux to SSYT to see if any interesting connections can be made.

Thus far viewing both ASMs and TSSCPPs as subsets of tournaments (or SSYT) has not produced the elusive bijection between ASMs and TSSCPPs. There is still hope, however, that further study of ASMs and TSSCPPs from the poset perspective may shed light on this area. See Figure~\ref{fig:bigpic} for the big picture of bijections and inclusions among ASMs, SSYT, tournaments, and TSSCPPs. 

\begin{figure}[htbp]
\centering
\includegraphics[scale=0.44]{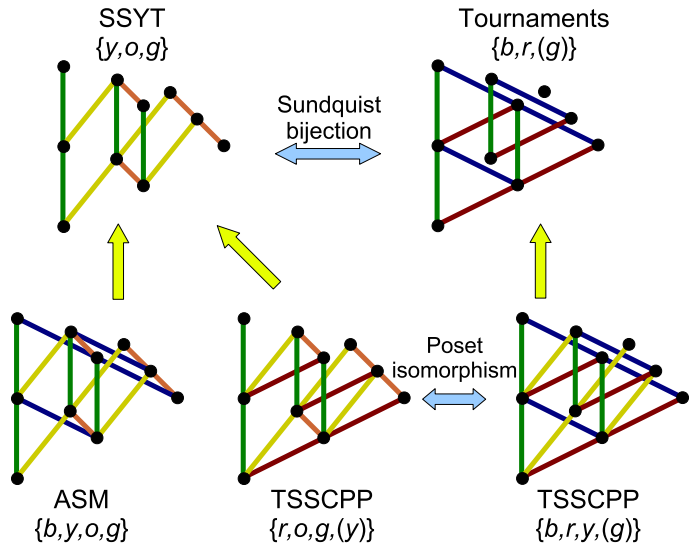}
\caption[Bijections and inclusions among three and four color posets]{The big picture of bijections and inclusions among the three and four color tetrahedral subposets. The two sided arrows represent bijections between sets of order ideals. The one sided arrows represent inclusions of one set of order ideals into another.}
\label{fig:bigpic}
\end{figure}

\section{The five color posets---intersections of ASMs and TSSCPPs}
\label{sec:five}
Alternating sign matrices and totally symmetric self-complementary plane partitions appear to be very different from one another. ASMs are square matrices of zeros, ones, and negative ones. TSSCPPs are highly symmetric stack of cubes in a corner. Thus it is strange to talk about objects in the intersection of ASMs and TSSCPPs. But when we view ASMs and TSSCPPs as order ideals in tetrahedral posets, that is subsets of the same set of elements subject to different constraints, we see that they are not so dissimilar after all. If we view them furthermore as staircase shape arrays subject to the inequality conditions corresponding to the edge colors as in Propositions~\ref{prop:asmbij} and~\ref{prop:tsscppbij} it is more transparent how one might take their intersection---simply by imposing the inequalities from both sets of arrays simultaneously.

There are two different ways to obtain nonisomorphic posets by choosing five of the six colors of edges in $T_n$. One way corresponds to the intersection of ASMs and TSSCPPs and the other corresponds to the intersection of two different sets of TSSCPPs inside $T_n$. The arrays $Y_n(\{r,b,o,(y),(g)\})$ can be thought of as the intersection of the arrays $Y_n(\{b,y,o,g\})$ corresponding to ASMs and the arrays $Y_n(\{r,b,y,(g)\})$ corresponding to TSSCPPs. The dual five color poset $Y_n(\{s,b,o,(y),(g)\})$ which has silver edges instead of red can be thought of similarly: as the intersection of $Y_n(\{b,y,o,g\})$ (ASMs) with $Y_n(\{s,b,y,(g)\})$ (TSSCPPs). We have not found a nice product formula for the number of order ideals of these posets, but one may yet exist. We have calculated the number of order ideals $|J(T_n(\{r,b,o,(y),(g)\}))|=|J(T_n(\{s,b,o,(y),(g)\}))|$ for $n=1$ to $7$ as $1, 2, 6, 26, 162, 1450, 18626$.

The other five color poset is $T_n(\{r,b,s,y,(g)\})$. The ASM poset $T_n(\{b,y,o,g\})$ is not a subset of this poset since the color orange is not included. So we can think of $T_n(\{r,b,s,y,(g)\})$ as an intersection of two manifestations of TSSCPPs inside $T_n$, but not as an intersection of ASMs and TSSCPPs. In fact, removing any of the colors except green from $T_n(\{r,b,s,y,(g)\})$ yields a TSSCPP poset, thus $T_n(\{r,b,s,y,(g)\})$ is the intersection of any two of those four TSSCPP posets. We have not found a nice product formula for the number of order ideals of this poset either, but we have calculated the number of order ideals $|J(T_n(\{r,b,s,y,(g)\}))|$ for $n=1$ to $6$ as $1, 2, 6, 28, 202, 2252$.


\section{The full tetrahedral poset and TSPPs}
\label{sec:six}

Throughout this paper we have seen how to begin with ${n+1\choose 3}$ poset elements and add ordering relations between them in certain ways until we have built the full tetrahedron $T_n$. In this section we prove Theorem~\ref{thm:6six} by showing that the order ideals of $T_n$ are in bijection with totally symmetric plane partitions inside an  $(n-1)\times (n-1)\times (n-1)$ box. We will then discuss connections with the conjectured $q$-enumeration of TSPPs and also a way to view TSPPs as an intersection of ASMs and TSSCPPs.

Totally symmetric plane partitions are plane partitions which are symmetric with respect to all permutations of the $x,y,z$ axes. Thus we can take as a fundamental domain the wedge where $x\ge y\ge z$. Then if we draw the lattice points in this wedge (inside a fixed bounding box of size $n-1$) as a poset with edges in the $x$, $y$, and $z$ directions, we obtain the poset $T_n$ where the $x$ direction corresponds to the red edges of $T_n$, the $y$ direction to the orange edges, and the $z$ direction to the silver edges. All other colors of edges in $T_n$ are induced by the colors red, silver, and orange. Thus TSPPs inside an $(n-1)\times (n-1)\times (n-1)$ box are in bijection with the order ideals of $T_n$ (see Figure~\ref{fig:tsppdomain}).

\begin{thm2110}
\begin{equation}
|J(T_n)|=\displaystyle\prod_{1\le i\le j\le n-1} \frac{i+j+n-2}{i+2j-2}=\displaystyle\prod_{1\le i\le j\le k\le n-1} \frac{i+j+k-1}{i+j+k-2}.
\end{equation}
\end{thm2110}
\begin{proof}
As discussed above, the order ideals of $T_n$ are in bijection with totally symmetric plane partitions inside an $(n-1)\times (n-1)\times (n-1)$ box. Thus by the enumeration of TSPPs in~\cite{STEMBRIDGE_TSPP} the number of order ideals $|J(T_n)|$ is given by the above formula.
\end{proof}

\begin{figure}[htbp]
\label{fig:tsppdomain}
\centering
\includegraphics[scale=0.5]{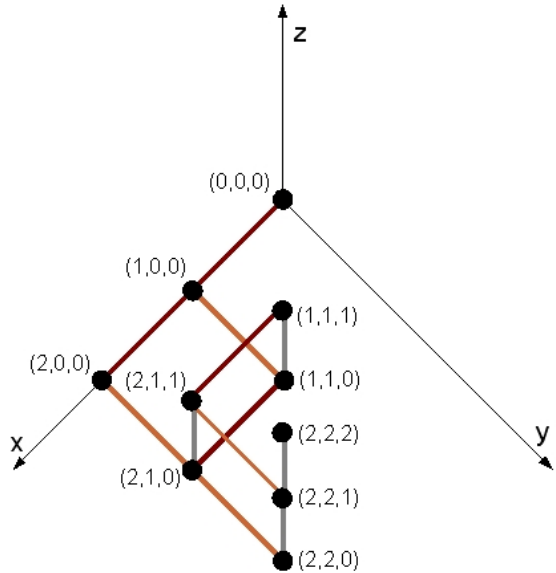}
\caption[The correspondence between the TSPP fundamental domain and $T_4$]{The correspondence between the TSPP fundamental domain $x\ge y\ge z$ and the tetrahedron $T_4$.}
\end{figure}

Stembridge enumerated TSPPs in~\cite{STEMBRIDGE_TSPP}, but no one has completely proved the $q$-version of this formula. If one could prove that 
$\displaystyle\prod_{1\le i\le j\le k\le n-1} \frac{[i+j+k-1]_q}{[i+j+k-2]_q}$ is the rank-generating function of $J(T_n)$, this would solve the last long-standing problem in the enumeration of the ten symmetry classes of plane partitions. Okada has written the $q$-enumeration of TSPPs as a determinant using lattice path techniques~\cite{OKADATSPP}, but this determinant has proved difficult to evaluate. Recently Kauers, Koutschan, and Zeilberger outlined a method for evaluating this determinant if given a large amount of time on the supercomputer~\cite{ZEILQTSPP}. 

Another implication of the poset perspective is that TSPPs can be seen as the intersection of the arrays corresponding to ASMs and TSSCPPs for a particular choice of the TSSCPP arrays.
\begin{proposition}
Totally symmetric plane partitions inside an $(n-1)\times (n-1)\times (n-1)$ box are in bijection with integer arrays corresponding to the intersection of $n\times n$ alternating sign matrices with totally symmetric self-complementary plane partitions inside a $2n\times 2n\times 2n$ box.
\end{proposition}
\begin{proof}
TSPPs inside an $(n-1)\times (n-1)\times (n-1)$ box are in bijection with the arrays $Y_n(\{r,s,o,(g),(y),(b)\})$ which are the intersection of the arrays $Y_n(\{b,y,o,g\})$ with either $Y_n(\{r,b,(g),s\})$ or $Y_n(\{r,y,(g),s\})$. $Y_n(\{b,y,o,g\})$ is in bijection with $n\times n$ ASMs while $Y_n(\{r,b,(g),s\})$ and $Y_n(\{r,y,(g),s\})$ are both in bijection with TSSCPPs inside a $2n\times 2n\times 2n$ box. 
\end{proof}
It is very strange to think of the set of totally symmetric plane partitions as being contained inside the set of totally symmetric self-complementary plane partitions, since TSSCPPs are TSPPs with the additional condition of being self-complementary. We must keep in mind that we are thinking about TSPPs in an $(n-1)\times (n-1)\times (n-1)$ box which is much smaller than the $2n\times 2n\times 2n$ box in which we are considering TSSCPPs.

\section{Extensions to trapezoids}
\label{sec:trap}
Zeilberger's original proof of the enumeration of ASMs in~\cite{ZEILASM} proved something more general. He proved that the number of gog trapezoids equals the number of magog trapezoids, which in our notation means that the number of distinct arrays obtained by cutting off the first $k$ southwest to northeast diagonals of the ASM arrays $Y_n(\{b,y,o,g\})$ equals the number of distinct arrays obtained by cutting off the first $k$ southwest to northeast diagonals of the TSSCPP arrays $Y_n(\{r,y,o,g\})$.
This was first conjectured by Mills, Robbins, and Rumsey~\cite{BRESSOUDBOOK}. 

Cutting off the first $k$ southwest to northeast diagonals of any  set of arrays $Y(S)$ for $S$ an admissible set of colors corresponds to cutting off a corner of the tetrahedral poset $T_n$. In particular, recall from Section~\ref{sec:intro} that $T_n$ can be thought of as the poset which results from beginning with the poset $P_n$, overlaying the posets $P_{n-1},P_{n-2},\ldots,P_3,P_2$ successively, and connecting each $P_{i}$ to $P_{i-1}$ by the orange, yellow, and silver edges. Also recall from Proposition~\ref{prop:arrays} that the chains made up of the green edges of the copy of $P_j$ inside $T_n$ determine the entries on the $(j-1)$st diagonal of the arrays $X_n(S)$ and likewise the arrays $Y_n(S)$. Therefore cutting off the first $k$ diagonals of the arrays in $Y_n(S)$ corresponds to removing the posets $P_2$, $P_3$,\ldots, $P_{k+1}$ and the edges adjacent to them from the Hasse diagram of $T_n$. Call the poset obtained by this process the trapezoidal poset and denote this poset as $T_n^k$. For $S$ an admissible subset of the colors, let $T_n^k(S)$ be $T_n^k$ with only the colors of edges in $S$ (thus $T_n^0(S)=T_n(S)$).

A natural question, then, is what the trapezoidal version of the theorems on the tetrahedral poset may be. Zeilberger's result in~\cite{ZEILASM} proves the following theorem, restated in our notation, relating the trapezoidal ASM and TSSCPP arrays. 
\begin{theorem}[Zeilberger]
\begin{equation}
|J(T_n^k(\{b,y,o,g\}))|=|J(T_n^k(\{r,y,o,g\}))|.
\end{equation}
\end{theorem}
Besides this theorem, nothing else is known about $T_n^k(S)$ for other admissible subsets of colors $S$. Perhaps the generalization of the results from this paper to the trapezoidal case will find relationships between more combinatorial objects and help to find a bijection between ASMs and TSSCPPs.

\section{Acknowledgements}
This work is based on research which is a part of the author's doctoral thesis at the University of Minnesota under the direction of Dennis Stanton. The author would like to thank Professor Stanton for the many helpful discussions and encouragement. The author would also like to thank the anonymous referee for helpful suggestions.



\end{document}